\documentclass[12pt,oneside]{amsart}

\usepackage{graphicx}              
\usepackage{amsmath}               
\usepackage{amsfonts}              
\usepackage{amsthm}                
\usepackage{amssymb}
\usepackage[a4paper]{geometry}
\usepackage{mathptmx}
\usepackage[scaled=.90]{helvet}
\usepackage[all]{xy}
\usepackage{color}
\usepackage{hyperref}
\usepackage[utf8]{inputenc}

\usepackage{comment}

\newtheorem{theorem}{Theorem}[section]
\newtheorem{lemma}[theorem]{Lemma}
\newtheorem{proposition}[theorem]{Proposition}
\newtheorem{corollary}[theorem]{Corollary}

\theoremstyle{definition}
\newtheorem{definition}[theorem]{Definition}
\newtheorem{example}[theorem]{Example}
\newtheorem{remark}[theorem]{Remark}

\DeclareMathOperator{\id}{id}

\DeclareMathOperator{\dd}{d}
\DeclareMathOperator{\T}{T}
\DeclareMathOperator{\Prop}{Prop}
\DeclareMathOperator{\pr}{pr}
\DeclareMathOperator{\interior}{int}
\DeclareMathOperator{\Bis}{Bis}
\DeclareMathOperator{\Diff}{Diff}
\DeclareMathOperator{\Inv}{Inv}
\DeclareMathOperator{\res}{res}

\newcommand{\RR}{\mathbb{R}}      

\newcommand{\smooth}{C^\infty}
\newcommand{\vsmooth}{C_{\textup{vS}}^\infty}
\newcommand{\fsmooth}{C_{\textup{fS}}^\infty}
\newcommand{\supp}{\operatorname{supp}}
\newcommand{\csupp}[2]{\operatorname{supp}({#1},{#2})}
\newcommand{\osupp}[2]{\operatorname{dif}({#1},{#2})}
\newcommand{\toto}{\ensuremath{\nobreak\rightrightarrows\nobreak}}
\newcommand{\Frechet }{Fr\'{e}chet }

\title[Strong topologies for smooth maps]{Strong topologies for spaces of smooth maps with infinite-dimensional target}
\author{Eivind Otto Hjelle}
\author{Alexander Schmeding}
\thanks{E-mail addresses: \href{mailto:eivindhj@stud.ntnu.no}{eivindhj@stud.ntnu.no} and \href{mailto:alexander.schmeding@math.ntnu.no}{alexander.schmeding@math.ntnu.no}}

\begin{document}

\begin{abstract}
 In this article we study two ``strong'' topologies for spaces of smooth functions from a finite-dimensional manifold to a (possibly infinite-dimensional) manifold modelled on a locally convex space. 
 Namely, we construct Whitney type topologies for these spaces and a certain refinement corresponding to Michor's $\mathcal{FD}$-topology. 
 Then we establish the continuity of certain mappings between spaces of smooth mappings, e.g.\ the continuity of the joint composition map.
 As a first application we prove that the bisection group of an arbitrary Lie groupoid (with finite-dimensional base) is a topological group (with respect to these topologies). \\
 For the reader's convenience the article includes also a proof of the folklore fact that the Whitney topologies defined via jet bundles coincide with the ones defined via local charts.
\end{abstract}

\maketitle


\medskip

(secondary)
\medskip

\textbf{Acknowledgment:} This research was part of StudForsk at NTNU, Trondheim. We acknowledge financial support by Olav Thon Stiftelsen.

\tableofcontents

\section*{Introduction and statement of results}

This paper gives a systematic treatment of two topologies for spaces of smooth functions from a finite-dimensional manifold to a (possibly infinite-dimensional) manifold modeled on a locally convex space.

In particular, we establish the continuity of certain mappings between spaces of smooth mappings, e.g.\ the continuity of the joint composition map.
As a first application we prove that the bisection group of an arbitrary Lie groupoid (with finite-dimensional base) is a topological group.
For the most part, these results are generalizations of well known constructions to spaces of smooth functions with infinite-dimensional range.
We refer to \cite{illman,michor,hirsch} for topologies on spaces of smooth functions between finite-dimensional manifolds.
\medskip

To understand these results of the present article, recall first the situation for spaces of smooth functions between finite-dimensional manifolds. 
For $0 \leq r \leq \infty$, let $C^r (M,N)$ denote the set of $r$-times continuously differentiable functions between manifolds $M$ and $N$
In the case where $r$ is finite, the standard choice for a topology on $C^r (M,N)$ is the well known Whitney $C^r$-topology (cf.\ \cite{illman,hirsch}). 
For $r=\infty$ and $M$ non-compact there are several choices for a suitable topology. 
One can for example choose the topology generated by the union of all Whitney $C^r$-topologies. 
We call this topology the strong $C^\infty$-topology and write $C^\infty_S (M,N)$ for the smooth functions with this topology.\footnote{The strong topology is in the literature often also called the ``Whitney $C^\infty$-topology''. Following Illman in \cite{illman}, we will not use this naming convention as it can be argued that the strong $C^\infty$-topology is not a genuine $C^\infty$-topology. See ibid. for more information.}
Note that each basic neighborhood of the strong $C^\infty$-topology allows one to control derivatives of functions only up to a fixed upper bound.
However, in applications one wants to control the derivatives of up to arbitrary high order (this is made precise in Section \ref{sect: vsTop}).
To achieve this one has to refine the strong topology, obtaining the \emph{very strong} topology\footnote{In \cite{michor} this topology is called the $\mathcal{D}$-topology.} in the process (cf.\ \cite{illman} for an exposition).
We denote by $\vsmooth (M,N)$ the smooth functions with the very strong topology and note that this topology is fine enough for many questions arising from differential topology.

Unfortunately, as is argued in \cite{michor} this topology is still not fine enough, if one wants to obtain manifold structures on $C^\infty (M,N)$ (and subsequently on the group of diffeomorphisms $\Diff (M)$).
Hence Michor constructed a further refinement of the very strong topology, called the $\mathcal{FD}$-topology. 
In the present paper, we call this topology the \emph{fine very strong} topology and denote the space of smooth functions with this topology by $\fsmooth (M,N)$.

Note that the topologies discussed so far coincide if the source manifold is compact. 
In fact in this case, all of these topologies coincide with the compact open $C^\infty$-topology (see e.g.\ \cite[Definition 5.1]{neeb}). 
The compact open $C^\infty$-topology for infinite-dimensional target manifolds is already well understood and has been used in many investigations, for example in infinite-dimensional Lie theory, e.g.\ \cite{glock1}.
Hence our investigation will only turn up new results for non-compact source manifolds and infinite-dimensional target manifolds.
\smallskip

We will now go into some more detail and explain the main results of the present paper.
Our aim is now to generalize the construction of the very strong and fine very strong topology to the set of smooth functions $C^\infty (M,X)$, where $X$ is a locally convex manifold.
Here smooth maps are understood in the sense of Bastiani's calculus \cite{bastiani} (often also called Keller's $C^r$-theory~\cite{keller}).
We refer to \cite{milnor1983,glock1,neeb} for streamlined expositions, but have included a brief recollection in Appendix \ref{calculus}.

Working in this framework we construct the very strong and the fine very strong topology for $C^\infty (M,X)$, where $M$ is finite-dimensional and $X$ is a locally convex manifold.
Our exposition mostly follows Illman's article \cite{illman} and we adapt his arguments to our setting. 
In particular, we describe the topology in terms of local charts as in \cite{illman} (cf.\ also \cite{hirsch}).
For finite-dimensional manifolds one can alternatively introduce the topology using jet bundles and it is well known that both approaches yield the same topology.
This fact seems to be a folklore theorem, but we were not able to locate a proof in the literature. 
As this fact is needed later on, a proof is given in Appendix \ref{folklore}.  
The advantage of the approach using local charts can be summarized as follows: Arguments and proofs often split into two distinct steps.
First one establishes a property of the function space topology only for the (easier) special case of vector space valued smooth mappings. 
Then a localization argument involving manifold charts allows one to establish the result for smooth maps between manifolds.

To our knowledge the topologies discussed in the present paper have so far only been studied for finite-dimensional manifolds. A topology somewhat similar to the very strong topology but for infinite-dimensional manifolds can be found in \cite[Section 41]{KM97}. Albeit the similar look, be aware that the jet bundles used in the construction are only manifolds in the inequivalent convenient setting of calculus. In particular, the topology in loc.cit.\ does not coincide with the one constructed here if $M$ is non-compact (cf.\ \cite[42.2 Remarks]{KM97}). We refer to Remark \ref{rem: topgeneral} for related topologies on function spaces between Banach manifolds.
For finite-dimensional manifolds, our construction recovers exactly the ones in the literature.  
We exploit this and recall that the set $\Prop (N,N) \subseteq C^\infty (M,N)$ of all proper maps is open in the very strong and the fine very strong topology.
Then one can establish continuity of certain composition mappings, in particular our results subsume the following theorem.\smallskip

\textbf{Theorem A} \emph{Let $M$, $N$ be finite-dimensional manifolds, $X$ and $Y$ be (possibly infinite-dimensional) manifolds. In the following, endow all function spaces either with the very strong or the fine very strong topology. Then 
 the joint composition 
  \begin{displaymath}
   \Gamma \colon \Prop (M,N) \times C^\infty (N,X) \rightarrow C^\infty (M,X) ,\quad  (f,g) \mapsto g\circ f
  \end{displaymath}
  is continuous.}
  
 \emph{Further, for any smooth map $h \colon X \rightarrow Y$, the pushforward $$h_* \colon C^\infty (M,X) \rightarrow C^\infty (M,Y) , \quad f \mapsto h\circ f$$ 
  is continuous.}

Having this theorem at our disposal, we construct an interesting class of topological groups: 
Suppose $\mathcal{G} = (G \toto M)$ is a Lie groupoid. This means that $G,M$ are smooth manifolds, equipped with submersions $\alpha,\beta \colon G\rightarrow M$ and an associative and smooth multiplication $G\times _{\alpha,\beta}G \rightarrow G$ that
 admits a smooth identity map $1 \colon M\rightarrow G$ and a smooth inversion $\iota\colon G\rightarrow G$. 
 Then the bisections $\Bis(\mathcal{G})$ of $\mathcal{G}$ are the sections
 $\sigma\colon M\rightarrow G$ of $\alpha$ such that $\beta \circ \sigma$ is a
 diffeomorphism of $M$. This becomes a group with respect to 
 \begin{equation*}
  (\sigma \star \tau ) (x) := \sigma ((\beta \circ \tau)(x))\tau(x)\text{ for }  x \in M.
 \end{equation*}
Many interesting groups from differential geometry such as diffeomorphism groups, automorphism groups and gauge transformations of principle bundles can be realised as bisection groups of suitable Lie groupoids.  
By construction $\Bis (\mathcal{G}) \subseteq C^\infty (M,G)$ and with respect to the topologies on the space of smooth functions we obtain the following.

\textbf{Theorem B} 
\emph{Let $\mathcal{G} = (G\toto M)$ be a Lie groupoid with finite-dimensional base $M$. Then $(\Bis (\mathcal{G}),\star)$ is a topological group with respect to the subspace topology induced by either the very strong or the fine very strong topology on $C^\infty (M,G)$.}

This result is a first step needed to turn the bisection group into an infinite-dimensional Lie group. 
In fact, it turns out that one can establish this result quite easily (see below) once Theorem B is available.
The key step to establish the applications mentioned below, is to work out the continuity of certain composition mappings (which has been done in Theorem A).
Then Proposition C and Theorem D below can be established using standard techniques from the literature.
In the present paper we will be only concerned with properties of the topology on function spaces. 
Hence the next results are stated without a proof. We provide only some references to the literature and hope to provide details in future work.

\textbf{Proposition C} \emph{
Let $M$ be a finite-dimensional manifold and $X$ be a possibly infinite-dimensional manifold which admits a local addition.\footnote{This is for example  satisfied if $X$ is a Lie group, see also \cite[Section 42.4]{KM97} for a definition of local additions and more examples.} 
Then $\fsmooth (M,X)$ can be turned into a manifold modeled on spaces of compactly supported sections of certain bundles.}

It turns out that once the space of smooth functions is endowed with the correct topology it is not hard to prove Proposition C.
More details and references to literature containing the necessary auxiliary facts can be found at the end of Section \ref{sect: vsTop}.
Proposition C generalizes \cite[Theorem 10.4]{michor} in so far as it admits arbitrary infinite-dimensional manifolds as target manifolds (whereas loc.cit.\ was confined to finite-dimensional targets). 
We remark that in \cite[42.4 Theorem]{KM97} the smooth functions $C^\infty (M,X)$ for $M$ and $X$ as in Proposition C have been endowed with a manifold structure in the inequivalent convenient setting of calculus.
However, following \cite[42.2 Remarks]{KM97} the topology on $C^\infty (M,X)$ used in the construction does not coincide with the fine very strong topology if $M$ is non-compact. 
Hence both constructions are inequivalent even if both $M$ and $X$ are finite-dimensional (and \( M \) is non-compact).

The manifold structure provided by Proposition C allows one to establish the Lie group structure for a general class of bisection groups. Adapting arguments from \cite{michor} and \cite{Schmeding2015} one can prove that 

\textbf{Theorem D} \emph{
The group of bisections of a Lie groupoid $\mathcal{G} = (G\toto M)$ with $M$ finite-dimensional and $G$ a Banach manifold\footnote{Assuming certain mild conditions on $G$ (i.e.\ an adapted local addition, cf. \cite{Schmeding2015}), it is not necessary to assume that $G$ is a Banach manifold.} is an infinite-dimensional Lie group. }

This generalizes the construction from \cite{Schmeding2015}, where the group of bisections of a Lie groupoid with \emph{compact} base was turned into an infinite-dimensional Lie group. 
Thus one obtains a conceptual approach to the Lie group structures of many groups which are of interest in differential geometry (e.g.\ automorphism groups and gauge transformation groups of principle bundles over a \emph{non compact} base).
Moreover, Theorem D is a crucial ingredient if one wants to extend the strong connection between Lie groupoids and infinite-dimensional Lie groups which was developed in \cite{SchmedingWockel15}.

\newpage
\section{The very strong topology}\label{sect: vsTop}
In this section, we introduce the \emph{very strong topology} on the space $ \smooth (M,X) $ of smooth maps from a finite-dimensional smooth manifold $ M $ to a possibly infinite-dimensional smooth manifold $ X $. 
The very strong topology allows us to control derivatives of smooth maps up to arbitrarily high order on certain families of compact sets.
This is a straightforward generalization of the very strong topology on the space of smooth maps between finite-dimensional manifolds as described in \cite{illman}.

\textbf{Notation and conventions.} We write $ \mathbb{N} := \lbrace 1,2,\dots \rbrace $ and $ \mathbb{N}_0 := \lbrace 0,1,\dots \rbrace $, and will only work with vector spaces over the field of real numbers $ \RR $. Finite-dimensional manifolds are always assumed to be $\sigma$-compact, i.e. a countable union of compact subspaces (which for finite-dimensional manifolds is equivalent to being second countable). We always endow \( \RR^n \) with the supremum norm $ \Vert \cdot \Vert_\infty $ unless otherwise stated. We define $ B_\epsilon^n (x) := \lbrace y \in \mathbb{R}^n : \Vert y - x \Vert_\infty < \epsilon \rbrace $. Notation and conventions regarding locally convex vector spaces, smooth maps, and infinite-dimensional manifolds is covered in Appendix \ref{calculus}. Typically, $ M $ and $ N $ will be finite-dimensional smooth manifolds, $ X $ a smooth manifold modeled on a locally convex vector space, and $ E $ a locally convex vector space.

\begin{definition}
	\label{norm}
	Let $E$ be a locally convex vector space, $p$ a continuous seminorm on $E$, $f \colon \mathbb{R}^m \to E $ smooth, $A \subseteq \mathbb{R}^m $ compact, $r \in \mathbb{N}_0 $, and $e_1, \dots, e_m $ the standard basis vectors in $\mathbb{R}^m$. Then define
	\begin{displaymath}
	\Vert f \Vert (r,A,p) = \sup \lbrace p(\dd^{(k)} f(a;\alpha)) : a \in A, \alpha \in \lbrace e_1, \dots , e_m \rbrace^k, 0 \leq k \leq r \rbrace .
	\end{displaymath}
\end{definition}
\begin{remark}
	\label{normremark}
	The symbol $ \dd^{(k)} f $ is defined in Definition \ref{Crmap}. Elsewhere in the literature, $ \dd^{(k)} f (x;y) = \dd^{(k)} f (x;y_1,\dots,y_k) $ is often denoted
	\begin{align*}
	\frac{\partial^k}{\partial y_k \cdots \partial y_1}f(x) && \mbox{or} && \frac{\partial}{\partial y} f(x),
	\end{align*}
	where $ y = (y_1,\dots,y_k) $.
	
	In the definition above we require $ \alpha \in \lbrace e_1,\dots,e_m \rbrace^k $. But for any $ \alpha \in B_1^n (0) $ and $ a \in A $ and $ k \leq r $ we have $ p (\dd^{(k)} f (a;\alpha)) \leq K \Vert f \Vert (r,A,p) $ for some constant $ K $ depending only on $ r $ and $ m $, by \eqref{schwarzrule} in Proposition \ref{chainruleprop}.
	
	If $ E = \RR^n $, any norm generates the topology on $ E $ and norms are in particular seminorms. By Proposition \ref{generatingfamilies}, the very strong topology is not affected if we always assume that the seminorm $p$ on $ \RR^n $ is the supremum norm $ \Vert\cdot \Vert_\infty $. In this case we simply write $ \Vert f \Vert (r,A) $ for $ \Vert f \Vert (r,A,\Vert\cdot \Vert_\infty ) $.
\end{remark}

\begin{lemma}[Triangle inequality]
	\label{triangleinequality}
	Let $ E,p,A,r $ be as in Definition \ref{norm}. Then the map
	\begin{displaymath}
		\Vert\cdot \Vert (r,A,p) : \smooth (\mathbb{R}^m,E) \to \mathbb{R}
	\end{displaymath}
	satisfies the triangle inequality. In fact it is a seminorm on $ \smooth (\mathbb{R}^m,E) $.
\end{lemma}
\begin{proof}
		Use linearity of $ d(-)(a,\alpha) $ for fixed $ (a, \alpha) $, and the fact that $ p $ satisfies the triangle inequality.
\end{proof}

\begin{definition}[Elementary neighborhood]
	\label{elementarynbh}
	Let $E$, $p$, and $r$ be as in Definition \ref{norm}, $M$ an $m$-dimensional smooth manifold, $ X $ a smooth manifold modeled on $ E $. Consider \( f \colon M \to X \) smooth, $(U,\phi)$ a chart on $M$, $ (V,\psi) $ a chart on $ X $, $ A \subseteq U $ compact such that $ f(A) \subseteq V $, and $ \epsilon > 0 $. Define
	\begin{align*}
	\mathcal{N}^r (f; A,(U,\phi),(V,\psi),p,\epsilon ) = \lbrace h &\in C^\infty (M,E) : \mbox{$ h(A) \subseteq V $ and } \\
	&\Vert \psi \circ h \circ \phi^{-1} - \psi \circ f \circ \phi^{-1} \Vert (r, \phi (A), p) < \epsilon \rbrace .
	\end{align*}
	We call this set an \emph{elementary ~$C^r$-neighborhood of }~$ f $ in ~$ \smooth (M,E) $.
\end{definition}

\textbf{Conventions for elementary neighborhoods}
If $ X = \RR^n $, we will assume that $ p $ is the supremum norm and omit the $ p $ when writing down the elementary neighborhoods.

When there is a canonical choice of charts for our manifolds, e.g.\ if $ X = E $ is a locally convex vector space, we omit the obvious charts when writing down elementary $ C^r $-neighborhoods.
Thus for $ f \colon M \to E $ we write e.g.\ $ \mathcal{N}^r (f;A,(U,\phi),p,\epsilon) := \mathcal{N}^r (f;A,(U,\phi),(X,\id ),p,\epsilon ) .$

\begin{remark}
	\begin{enumerate}
	 \item The conditions $ f(A) \subseteq V $ and $ h(A) \subseteq V $ ensure that the map $ \psi \circ h \circ \phi^{-1} - \psi \circ f \circ \phi^{-1} $ makes sense. 
	Further, the conditions enable us to control the open sets into which a (given) compact set is mapped, i.e.\ the kind of control provided by the well known compact open topology (cf.\ \cite[Definition I.5.1]{neeb}).
        Indeed, by restricting to elementary $C^0$-neighborhoods, one would recover a subbase of the compact open topology on $C^\infty (M,X)$.
	 \item We define elementary neighborhoods only for finite-dimensional source manifolds as the seminorms in Definition \ref{norm} make only sense for these manifolds. 
	 Compare Remark \ref{rem: topgeneral} for more information on alternative approaches to the topology which avoid this problem. 
	\end{enumerate}
\end{remark}

We now define what will become the basis sets in the very strong topology on $ \smooth (M,X) $. 
\begin{definition}[Basic neighborhood]
	\label{basicnbh}
	Let $ f \colon M \to X $ be a smooth map from a finite-dimensional smooth manifold $M$ to a smooth manifold $ X $ modeled on a locally convex vector space $E$. A \emph{basic neighborhood of $f$ in ~$ \smooth (M,X) $} is a set of the form
	\begin{displaymath}
	\bigcap_{i \in \Lambda } \mathcal{N}^{r_i} (f; A_i,(U_i, \phi_i),(V_i,\psi_i),p_i, \epsilon_i),
	\end{displaymath}
	where $ \Lambda $ is a possibly infinite indexing set, for all \( i \) the other parameters are as in Definition \ref{elementarynbh}, and $ \lbrace A_i \rbrace_{i \in \Lambda} $ is locally finite. We call $ \lbrace A_i \rbrace_{i \in \Lambda} $ the \emph{underlying compact family} of the neighborhood.
\end{definition}
Without loss of generalization, \( \Lambda = \mathbb{N} \), since every locally finite family over a \( \sigma \)-compact space is countable.

As Proposition \ref{basicnbhsisbasis} show, the basic neighborhoods in ~$ \smooth (M,X) $ form a basis for a topology on ~$ \smooth (M,X) $. In order to prove the proposition we need the following lemma.
\begin{lemma}
	\label{trianglelemma}
	Let $ f : M \to X $ be smooth, and ~$ g \in \mathcal{N} := \mathcal{N}^r (f; A, (U,\phi),(V,\psi),p,\epsilon) $. 
	Then there exists $ \epsilon' > 0 $ such that ~$ \mathcal{N}' := \mathcal{N}^r (g; A, (U, \phi ),(V,\psi), p, \epsilon') \subseteq \mathcal{N} $.
\end{lemma}
\begin{proof}
	For $ h,\tilde{h} \in \smooth (M,X) $ with $ h(A),\tilde{h}(A) \subseteq V $, let
	\begin{displaymath}
		d(h,\tilde{h}) = \Vert \psi \circ \tilde{h} \circ \phi^{-1} - \psi \circ h \circ \phi^{-1} \Vert (r,\phi(A),p).
	\end{displaymath}
	Note that $ d $ satisfies the triangle inequality by Lemma \ref{triangleinequality}, and that $ h \in \mathcal{N} $ is equivalent to $ d(f,h)<\epsilon $.
	
	Set $ \epsilon' = \epsilon - d(f,g) $, and let $ \mathcal{N}' $ be as in the statement of the lemma. If $ h \in \mathcal{N}' $, then 
	\begin{displaymath}
	d(f,h) \leq d(f,g) + d(g,h) < d(f,g) + (\epsilon - d(f,g)) = \epsilon.
	\end{displaymath}
	Hence $ h \in \mathcal{N} $, and $ \mathcal{N}' \subseteq \mathcal{N} $.
\end{proof}
\begin{proposition}
	\label{basicnbhsisbasis}
	Let $ \mathcal{U} $ and $ \mathcal{U}' $ be basic neighborhoods of $ f $ and $ f' $ in $ \smooth (M,X) $, respectively. If $ g \in \mathcal{U} \cap \mathcal{U}' $, then there exists a basic neighborhood $ \mathcal{V} $ of $ g $ such that $ \mathcal{V} \subseteq \mathcal{U} \cap \mathcal{U}' $.
	
	Hence the basic neighborhoods form a basis for a topology on $ C^\infty (M,E) $, called \emph{the very strong topology on $ C^\infty (M,E) $}.
\end{proposition}
\begin{proof}
	We may write
	\begin{align*}
	\mathcal{U} = \bigcap_{i \in \Lambda} \mathcal{N}_i && \mbox{and} && \mathcal{U}' = \bigcap_{j \in \Lambda'} \mathcal{N}'_j 
	\end{align*}
	for some sets $ \Lambda $ and $ \Lambda' $, where $ \mathcal{N}_i $ and $ \mathcal{N}'_i $ are elementary neighborhoods of $ f $ and $ f' $, respectively.
	For all $ i \in \Lambda $ and $ j \in \Lambda $ choose as in Lemma \ref{trianglelemma} elementary neighborhoods $ \mathcal{M}_i $ and $ \mathcal{M}'_j $ of $ g $ such that $ \mathcal{M}_i \subset \mathcal{N}_i $ and $ \mathcal{M}'_j \subset \mathcal{N}'_j $. Then 
	\begin{displaymath}
	\mathcal{V} := \left( \bigcap_{i \in \Lambda} \mathcal{M}_i \right) \cap \left( \bigcap_{j \in \Lambda'} \mathcal{M}'_i \right) \subseteq \mathcal{U} \cap \mathcal{U}'. 
	\end{displaymath} 
	It remains to check that $ \mathcal{V} $ is in fact a basic neighborhood of $ g $. The set $ \mathcal{V} $ is a basic neighborhood of $ g $ provided that the underlying compact family of $ \mathcal{V} $ is locally finite. This is indeed the case since the underlying compact families of $ \mathcal{U} $ and $ \mathcal{U}' $ are locally finite and finite unions of locally finite families are locally finite.
\end{proof}
The preceding proposition justifies the following definition.
\begin{definition}[Very strong topology]
	The \emph{very strong topology on $ \smooth (M,X) $} is the topology on $ \smooth (M,X) $ with basis the basic neighborhoods in $ \smooth (M,X) $.\\
	The set $ \smooth (M,X) $ equipped with the very strong topology will be denoted by $ \vsmooth (M,X) $.
\end{definition}

\begin{remark}
 We will work later on with $\vsmooth (M,E)$, where $E$ is a locally convex space. To this end, we considered $E$ as a manifold with the canonical atlas given by the identity.
 This may seem artificial at first glance as one in principle needs to take all ``manifold charts'' of $E$ into account.
 Note however that by Lemma \ref{lem: atlaschoice} the very strong topology on $C^\infty (M,E)$ is generated by all basic neighborhoods of the form 
 \begin{equation}\label{loc: charts}
  \bigcap_{i \in \Lambda } \mathcal{N}^{r_i} (f; A_i,(U_i, \phi_i),(E,\id_E),p_i, \epsilon_i),
 \end{equation}
 i.e.\ it suffices to consider elementary neighborhoods with respect to the identity chart.
 Hence the topology on $C^\infty (M,E)$ is quite natural.
 
 Similarly for $C^\infty (\mathbb{R}^n,E)$ the charts $(U_i,\phi_i)$ in \eqref{loc: charts} can be replaced by $(\mathbb{R}^n,\id_{\mathbb{R}^n})$ by Lemma \ref{lem: atlaschoice2}. 
 In the following, we will always assume that our elementary and basic neighborhoods are constructed with respect to the identity if one (or both) of the manifolds are a locally convex space.  
\end{remark}

\begin{remark}
	There are other well-known topologies on $ \smooth (M,X) $. The \emph{strong} topology (or \emph{Whitney} $\smooth$-topology) and the \emph{compact open} $\smooth$-topology (or \emph{weak} topology) have as bases neighborhoods of the form described in Definition \ref{basicnbh}, with some additional restrictions. For the strong topology the collection $\lbrace  r_i \rbrace_{i \in \Lambda} $ of indices giving differentiation order is bounded, and for the compact open $\smooth$-topology we require that the indexing set $ \Lambda $ is finite. 
	
	The very strong topology is finer than the strong topology which is finer than the compact open $\smooth$-topology, and in the case that $ M $ is compact all of these topologies coincide (since every locally finite family meets a compact set only finitely many times). We refer the reader to section 2.1 in \cite{hirsch} for information about the strong and compact open $\smooth$ topologies in the case that $ X $ is finite-dimensional. A comparison of the strong topology and the very strong topology can be found in the introduction of \cite{illman}.
	
	Since the very strong topology is finer than the strong topology, subsets of $ \smooth (M,X) $ that are open in the strong topology are also open in the very strong topology. \cite[Section 2.1]{hirsch} has several results stating that certain subsets of $ \smooth (M,N) $ are open in the strong topology, consequently also in the very strong topology. In particular, the set $ \Prop (M,N) $ of proper smooth maps is open in $ \vsmooth (M,N) $. We write $ \Prop_{\textup{vS}} (M,N) $ for the subspace $ \Prop (M,N) $ of $ \vsmooth (M,N) $ equipped with the subspace topology.
\end{remark}

\begin{remark}\label{rem: topgeneral}
One can also define the very strong topology on the space $ \smooth (X,Y) $ where $ X $ and $ Y $ are Banach manifolds (i.e. modeled on Banach spaces). 
To this end one needs to redefine the seminorms generating the topology, which in the vector space case will take the following form:

If \( X,Y \) are Banach spaces, \( f \colon X \to Y \) smooth, \( r \in \mathbb{N}_0 \), and \( A \subseteq X \) compact define
\begin{equation}\label{semi:Frechet}
	\Vert f \Vert (r,A) = \sup \left\lbrace \Vert D^k f (x) \Vert_Y : \mbox{ \( 0 \leq k \leq r \) and \( x \in A \)} \right\rbrace,
\end{equation}
where \( D^k f \) denotes the \( k \)-th \emph{\Frechet  derivative} of \( f \).

Here we use that every smooth Bastiani map is also smooth in the sense of \Frechet differentiability by \cite[Lemma 2.10]{milnor1983}. It is easy to see that all statements made on elementary neighborhoods in the present section remain valid. Hence we obtain a very strong topology on smooth functions between Banach manifolds. 

	Note that one can prove as in Appendix \ref{folklore} that the ``very strong topology'' constructed with respect to the seminorms \eqref{semi:Frechet} induces again the (original) very strong topology on $ \smooth (X,Y) $ if $ X $ is finite-dimensional. 
	Unfortunately, for an infinite-dimensional Banach manifold $ X $ this topology does not allow us to control the behavior of functions ``at infinity'' (or anywhere for that matter since compact subsets of infinite-dimensional  Banach spaces have empty interior).
To see this recall that manifolds modeled on infinite-dimensional Banach space don't have a locally finite compact exhaustion by the Baire category theorem. 

	Recall however, that one can define a Whitney $ \smooth $-topology for $X,Y$ Banach manifolds via jet bundles (see e.g.\ \cite{michor,KM97} or Appendix \ref{folklore} for a short exposition). As shown in \cite[Chapter 9]{marg}, this topology then allows one to control the behavior of a function on all of $ X $. 
	The key difference is that the Whitney topology defined in this way controls the behavior of jets on locally finite families of \textbf{closed} sets. 
	Obviously, one can not hope to describe it via the seminorms as the existence of the suprema in the seminorms is tied to the compactness of the sets. Even worse, for an infinite-dimensional manifold $ X $ and $Y=E$ a locally convex space, the largest topological vector space contained in $\smooth(X,E)$ with respect to this topology is trivial (cf. \cite[437]{KM97}). For these reasons we work exclusively with the very strong topology for finite-dimensional source manifolds.
\end{remark}

\textbf{Additional facts about the very strong topology.} Sometimes it is convenient to assume that the continuous seminorms $ p $ used in constructing very strong neighborhoods are of a certain form, as we have already remarked. There is no loss of generality in making such assumptions if the family of seminorms that we restrict to is ``big enough''.

\begin{proposition}
	\label{generatingfamilies}
	Let $ M $ be a finite-dimensional smooth manifold and $ X $ a smooth manifold modeled on a locally convex vector space $ E $. Suppose $ \mathcal{P} $ is a generating family of seminorms for $ E $ (see Definition \ref{seminormfamilydef}). 
	
	If we replace every instance of ``$p$ is a continuous seminorm on $ E $'' in the definitions and results earlier in this section with ``$ p \in \mathcal{P} $'', then the resulting very strong topology on $ \smooth (M,X) $ is unaffected.
\end{proposition}
\begin{proof}
	Let $ \mathcal{T} $ be the very strong topology on $ \smooth (M,X) $ constructed with respect to all continuous seminorms on $ E $, and let $ \mathcal{T}' $ be the very strong topology on $ \smooth (M,X) $ obtained by restricting to seminorms in $ \mathcal{P} $. Then $ \mathcal{T}' $ is obviously coarser than $ \mathcal{T} $ since every $ p \in \mathcal{P} $ is continuous, so it suffices to show that $ \mathcal{T} $ is coarser than $ \mathcal{T}' $. This will be the case if for every basic $ \mathcal{T} $-very strong neighborhood $ \mathcal{U} = \bigcap_{i\in\Lambda} \mathcal{N}_i $ of $ f \in \smooth (M,X) $, where each $ \mathcal{N}_i $ is an elementary $ \mathcal{T} $-very strong neighborhood $$ \mathcal{N}_i = \mathcal{N}^{r_i} (f;A_i,(U_i,\phi_i),(V_i,\psi_i),p_i,\epsilon_i ), $$ there exists a basic $ \mathcal{T}'$-very strong neighborhood $ \mathcal{U}' $ of $ f $ such that $ \mathcal{U}' \subseteq \mathcal{U} $.
	
	Fix $ i \in \Lambda $. By \eqref{generating family criterion} in Proposition \ref{locally convex prop} there exist $ n_i \in \mathbb{N} $ and $ p_{i,1},\dots,p_{i,n_i} \in \mathcal{P} $ and $ c_i > 0 $ such that $ p_i \leq c_i \sup_{1\leq j \leq n_i} p_{i,j} $. And then
	\begin{align*}  
	\mathcal{V}_i := \bigcap_{j=1}^{n_i} \mathcal{N}^{r_i} \left( f;A_i,(U_i,\phi_i),(V_i,\psi_i),p_{i,j}, \frac{ \epsilon_i }{2c_i} \right) \subseteq \mathcal{N}_i.
	\end{align*}
	Indeed, if $ g \in \mathcal{V}_i $, then for $ a \in A_i $, $ \alpha \in \lbrace e_1,\dots,e_m \rbrace^k $, $ 0 \leq k \leq r_i $, and $ 0 \leq j \leq n_i $, we have $$ c_i p_{i,j} ( \dd (\psi_i \circ g \circ \phi_i^{-1} - \psi_i \circ f \circ \phi_i^{-1})^{(k)}(a,\alpha) ) < \frac{ \epsilon_i}{2}, $$ which together with $ p_i \leq c_i \sup p_{i,j} $ clearly implies that $ g \in \mathcal{N}_i $.
	
	Now set $ \mathcal{U}' := \bigcap_{i \in \Lambda} \mathcal{V}_i $. This is a basic $ \mathcal{T}' $-very strong neighborhood of $ f $ such that $ \mathcal{U}' \subseteq \mathcal{U} $.
\end{proof}

The following lemma is useful when constructing certain basic neighborhoods. The proof given here is fairly detailed, but throughout the remainder of this text the details of similar arguments will be omitted.
\begin{lemma}
	\label{hackinglemma}
	Let \( M \) be a finite-dimensional smooth manifold, \( X \) a locally convex manifold, and \( f \colon M \to X \) a smooth map. Suppose \( \lbrace K_n \rbrace_{n \in \mathbb{N}} \) is a locally finite family of compact subsets of \( M \). Then there exist families of charts \( \lbrace (V_i,\psi_i) \rbrace_{i\in\mathbb{N}} \) for \( X \) and \( \lbrace (U_i,\phi_i) \rbrace_{i\in\mathbb{N}} \) for \( M \), and a locally finite family \( \lbrace A_i \rbrace_{i\in\mathbb{N}} \) of compact subsets of \( M \) such that 
	\begin{enumerate}
		\item \( \bigcup_{i\in\mathbb{N}} A_i = \bigcup_{n \in \mathbb{N}} K_n \),
		\item \( A_i \subseteq U_i \) for all \( i \in \mathbb{N} \),
		\item \( f(U_i) \subseteq V_i \) for all \( i \in \mathbb{N} \).
	\end{enumerate}
\end{lemma}
\begin{proof}
	Fix \( n \in \mathbb{N} \). For every \( x \in K_n \) choose a chart \( (V_{n,x},\psi_{n,x}) \) around \( f(x) \) and a chart \( (U_{n,x},\phi_{n,x}) \) around \( x \). By shrinking \( U_{n,x} \) we may assume that \( f(U_{n,x}) \subseteq V_{n,x} \). Since \( M \) is locally compact there exists a compact neighborhood \( A_{n,x}' \) around \( x \) such that \( A_{n,x}' \subseteq U_{n,x} \). Now set \( A_{n,x} = K_n \cap A_{n,x}' \). By compactness of \( K_n \) there exist finitely many \( x_{n,1},\dots,x_{n,k_n} \in K_n \) such that \( \lbrace A_{n,x_{n,j}} \rbrace_{i=1}^{k_n}  \) covers \( K_n \). 
	
	The families \( \lbrace (V_{n,x_{n,i}},\psi_{n,x_{n,i}}) \rbrace_{n,i} \), \( \lbrace (U_{n,x_{n,i}},\phi_{n,x_{n,i}}) \rbrace_{n,i} \) and \( \lbrace A_{n,x_{n,i}} \rbrace_{n,i} \) have the desired properties. By relabeling the indices we can take the indexing set to be \( \mathbb{N} \).
\end{proof}

\begin{lemma}
	\label{biglemma}
	Let $ M $ be a finite-dimensional smooth manifold, $ X $ a smooth manifold modeled on a locally convex vector space $ E $, and let \( U \subseteq M \) and \( V \subseteq X \) be open subsets. Consider the subspace \( \smooth_{\text{vS,sub}} (U,V) := \left\lbrace f \in \smooth (M,X) : f(U) \subseteq V \right\rbrace \subseteq \vsmooth (M,X) \).
	\begin{enumerate}
		\item \( \smooth_{\text{vS,sub}} (U,V) \) is an open subset of \( \vsmooth (M,X) \).
		\item The restriction \( \res_{\text{vS}} \colon \smooth_{ \text{vS,sub}} (U,V) \to \vsmooth (U,V) \) is continuous.
		\item If $f \in C^\infty (U,V)$ and $\mathcal{N}^r (f; A,(U_\phi,\phi),(V_\psi,\psi),p,\epsilon )$ is an elementary neighborhood of $f$ such that $\psi (V_\psi)$ is a convex set, then there exists $g \in C^\infty (M,X)$ with
			  \begin{equation}\label{eq: elres}
		   \res_{\text{vS}}^{-1} (\mathcal{N}^r (f; A,(U_\phi,\phi),(V_\psi,\psi),p,\epsilon )) = \mathcal{N}^r (g; A,(U_\phi,\phi),(V_\psi,\psi),p,\epsilon )
		  \end{equation}
	\end{enumerate}
\end{lemma}
\begin{proof}
	\begin{enumerate}
		\item Suppose \( f \in \smooth(U,V) \).
			Since $ U $ is an open subset of $ M $, the subspace \( U \) is metrizable and locally compact, hence $ \sigma $-compact.
			Combining Lemma \ref{dugundjifact} and Lemma \ref{hackinglemma}, we find a locally finite exhaustion \( \lbrace A_n \rbrace_{n \in \mathbb{N}} \) of \( U \) by compact sets, charts \( \lbrace (U_n,\phi_n) \rbrace_{n \in \mathbb{N}} \) for \( M \), and charts \( \lbrace (V_n,\psi_n) \rbrace_{n \in \mathbb{N}} \) for \( X \) such that \( A_n \subseteq U_n \) and \( f(A_n) \subseteq V_n \) for all \( n \in \mathbb{N} \).
			Since \( f(A_n) \subseteq V \), shrink the \( V_n \) if necessary such that \( V_n \subseteq V \) for all \( n \in \mathbb{N} \) (while still \( f( A_n ) \subseteq V_n \)).
			Take any continuous seminorm $ p $ on $ E $ and define 
			\[ \mathcal{U} = \bigcap_{n \in \mathbb{N}} \mathcal{N}^0 \left( f; A_n,(U_n,\phi_n),(V_n,\psi_n),p,1 \right). \] 
			If $ g \in \mathcal{U} $, then $ g(A_n) \subseteq V_n \subseteq V $ for all $ n \in \mathbb{N} $, from which it follows that $ g(U) = g \left( \bigcup A_n \right) \subseteq V $. So $ \mathcal{U} $ is a neighborhood of $ f $ in $ \vsmooth (M,X) $ such that $ \mathcal{U} \subseteq \smooth (U,V) $.
		\item Take an arbitrary basic neighborhood
			\[
				\mathcal{U} = \bigcap_{i\in\Lambda} \mathcal{N}^{r_i} \left( f;A_i,(U_i,\phi_i),(V_i,\psi_i),p_i,\epsilon_i \right)
			\]
			in \( \vsmooth (U,V) \). We will show that given \( g \in \res_{\text{vS}}^{-1} (\mathcal{U}) \), there exists a basic neighborhood \( \mathcal{V} \) of \( g \) in \( \vsmooth (M,X) \) such that \( \mathcal{V} \subseteq \res_{\text{vS}}^{-1} (\mathcal{U}) \), i.e. \( \res_{\text{vS}}^{-1} (\mathcal{U}) \) is open. By Lemma \ref{basicnbhsisbasis} there are elementary neighborhoods of $\res_{\text{vS}} (g)$ such that 
			\[
				\mathcal{V} = \bigcap_{i\in\Lambda} \mathcal{N}^{r_i} \left( \res_{\text{vS}} (g) ;A_i,(U_i,\phi_i),(V_i,\psi_i),p_i,\delta_i \right).
			\]
			is contained in $\mathcal{U}$. 
			Now clearly $\bigcap_{i\in\Lambda} \mathcal{N}^{r_i} \left( g ;A_i,(U_i,\phi_i),(V_i,\psi_i),p_i,\delta_i \right)$  is contained in $\res_{\text{vS}}^{-1} (\mathcal{V}) \subseteq \res_{\text{vS}}^{-1} (\mathcal{U})$.
 		\item Since $M$ is finite-dimensional, whence paracompact, we can choose a neighborhood $W$ of $A$ and a smooth cutoff function $\rho \colon M \rightarrow \RR$ with $\rho|_{W} \equiv 1$ and $\rho_{M \setminus f^{-1} (V_\psi) \cap U_\phi}\equiv 0$.
 		Composing with a suitable translation, we may assume without loss of generality that $\psi(V_\psi)$ is a convex $0$-neighborhood.
 		Suppressing the translation we can thus define $$g \colon M \rightarrow X , \quad x \mapsto \begin{cases}
 		                                                             \psi^{-1} ( \rho(x) \cdot \psi \circ f (x)) & \text{if } x \in U_\phi, \\
 		                                                             \psi^{-1} (0) & \text{else}.
 		                                                            \end{cases}
		      $$
		Now as $g|_W = f|_W$ the identity \eqref{eq: elres} is satisfied.    \qedhere
	\end{enumerate}
\end{proof}

\section{Composition of maps in the very strong topology}\label{sect: compo}
Throughout this section, $ M $ and $ N $ are finite-dimensional smooth manifolds and $ X $ denotes a smooth manifold modeled on a locally convex vector space $ E $. 

It is a desirable property of the very strong topology on \( \smooth (M,X) \) that composition
\begin{align*}
	\Gamma : \vsmooth (M,N) \times \vsmooth (N,X) &\to \vsmooth (M,X) \\
	(f,h) &\mapsto h \circ f
\end{align*}
is continuous. But this is not the case in general, a counterexample can be found in Example \ref{counterexample}. However, the restriction of the composition map
\begin{displaymath}
\Gamma \colon \Prop_{\textup{vS}} (M,N) \times \vsmooth (N,X) \to \vsmooth (M,X)
\end{displaymath}
\textit{is} continuous, where $ \Prop_{\textup{vS}} (M,N) $ denotes the subspace of $ \vsmooth (M,N) $ consisting of all the proper maps. This is precisely what Theorem \ref{compiscont} says, and proving it is the main goal of this section.

As we will see, the crucial property of proper maps needed is the fact that if $ \lbrace A_i \rbrace $ is a locally finite family of subsets of $ M $ and $ f : M \to N $ is proper, then $ \lbrace f(A_i) \rbrace $ is locally finite. This will enable us to choose for a basic neighborhood $ \mathcal{V} $ of some composition $ h \circ f $ in $ \vsmooth (M,X) $ basic neighborhoods $ \mathcal{U} $ and $ \mathcal{U}' $ of $ f $ and $ h $, respectively, such that $ \Gamma (\mathcal{U} \times \mathcal{U}' ) \subseteq \mathcal{V} $. The challenge is to choose $ \mathcal{U}' $ in such a way that the underlying compact family of the neighborhood is locally finite.

We now give the promised counterexample to the statement that composition of maps in the very strong topology is continuous in general. This example is inspired by the proof of \cite[Proposition 2.2(b)]{counterglock}.
\begin{example}
	\label{counterexample}
	The composition map
	\begin{align*}
		\Gamma \colon \vsmooth (\mathbb{R},\mathbb{R}) \times \vsmooth (\mathbb{R},\mathbb{R}) \to \vsmooth (\mathbb{R},\mathbb{R}), \quad (f,h) \mapsto h \circ f
	\end{align*}
	is not continuous.
\end{example}
\begin{proof}
	Note that for every basic neighborhood $ \mathcal{U} $ of $ f \in \vsmooth (\mathbb{R},\mathbb{R}) $ there exists a basic neighborhood $ \mathcal{U}' $ of $ f $ with underlying compact family $ \lbrace [2n-1,2n+1] \rbrace_{n\in\mathbb{Z}} $ such that $ f \in \mathcal{U}' \subseteq \mathcal{U} $, since each compact interval $ [2n-1,2n+1] $ intersects only finitely many sets belonging to the locally finite underlying compact family of $ \mathcal{U} $.
	
	To show discontinuity of $ \Gamma $ it suffices to show discontinuity at $ (0,0) $. Let $ \mathcal{V} $ be the basic neighborhood of 0 given by $$ \mathcal{V} := \bigcap_{n\in\mathbb{N}} \mathcal{N}^{n} (0;[2n-1,2n+1],1) . $$ We will show that for any pair of basic neighborhoods
	\begin{align*}
	\mathcal{U} = \bigcap_{n\in\mathbb{Z}} \mathcal{N}^{r_n} (0;[2n-1,2n+1],\epsilon_n), \\
	\mathcal{U}' = \bigcap_{n\in\mathbb{Z}} \mathcal{N}^{r_n'} (0;[2n-1,2n+1],\epsilon_n'),
	\end{align*}
	there exists a pair of functions $ (f,h) \in \mathcal{U}' \times \mathcal{U} $ such that $ h \circ f \notin \mathcal{V} $. 
	
	Construct $ h \in \smooth (\mathbb{R},\mathbb{R}) $ such that in a neighborhood of $ 0 $, $ h $ is given by the equation $ h(x) = x^{r_0 + 1} $, and such that $ \supp h \subseteq ]-1,1[ $. For some sufficiently small $ k > 0 $ we will have $ kh \in \mathcal{U} $. For every $ m \in \mathbb{N} $ define $$ h_m (x) := \frac{k}{m^{r_0}} h(mx) $$ and note that $ h_m \in \mathcal{U} $, since $$ | h_m^{(j)} (x) | = \frac{k m^j}{m^{r_0}} | h^{(j)} (mx) | \leq k | h^{(j)}(mx) | < \epsilon_0 $$ for $ j \leq r_0 $, where we use the notation \( g^{(j)} (y) = \dd^{(j)} g(y;1,\dots,1) \) for smooth maps \( g \colon \RR \to \RR \).
	
	Let $ 2n \geq r_0 + 1 $ and construct $ \tilde{f} \in \smooth (\mathbb{R},\mathbb{R}) $ such that $ \tilde{f}(x) = x - 2n $ in a neighborhood of $ 2n $ and $ \supp \tilde{f} \subseteq ]2n-1,2n+1[ $. Then for some sufficiently small $ s > 0 $ we have $ f := s \tilde{f} \in \mathcal{U}' $.
	
	So far we have a sequence $ \lbrace h_m \rbrace_{m\in\mathbb{N}} \subset \mathcal{U} $ and $ f \in \mathcal{U}' $. By construction, $ h_m \circ f (x) = kms^{r_0+1} (x-2n)^{r_0+1} $ in a neighborhood of $ 2n $. Hence $$ |(h_m \circ f)^{(r_0+1)}(2n)| = kms^{r_0+1} (r_0+1)! \geq 1 $$ for large enough $ m $, in which case $ h_m \circ f \notin \mathcal{V} $.
\end{proof}
Having given the example above, we return our focus to the main task of the section, which is proving Theorem \ref{compiscont}. Leading up to the theorem is a sequence of lemmata. 

Although we are actually interested in mapping spaces between manifolds, we first give a lemma that only applies to vector spaces. In a sense, this lemma resolves the main difficulty, and generalizing to manifolds is only a matter of dealing with charts.
\begin{lemma}
	\label{realdomains}
	Consider the composition map
	\begin{displaymath}
		\Gamma : \vsmooth (\mathbb{R}^m,\mathbb{R}^n) \times \vsmooth (\mathbb{R}^n,E) \to \vsmooth (\mathbb{R}^m,E),\quad (f,h) \mapsto h \circ f.
	\end{displaymath}
	Let $ (f,h) \in \vsmooth (\mathbb{R}^m,\mathbb{R}^n) \times \vsmooth (\mathbb{R}^n,E) $, and consider an arbitrary elementary neighborhood
	$	\mathcal{N} = \mathcal{N}^r (h\circ f;A,p,\epsilon) \subseteq \vsmooth (\mathbb{R}^m,E)$
	of $ h \circ f $. For all compact neighborhoods \( A' \) of \( f(A) \) there exist \( \delta, \delta' > 0 \) such that the elementary neighborhoods
	\begin{align*}
		\mathcal{M} = \mathcal{N}^r (f;A,\delta) && \mbox{and} && \mathcal{M}' = \mathcal{N}^r (h;A',p,\delta')
	\end{align*}
	satisfy $ \Gamma (\mathcal{M} \times \mathcal{M}') \subseteq \mathcal{N} $.
\end{lemma}
\begin{proof}
	Let $ A' $ be any compact neighborhood of $ f(A) $. We proceed in several steps.
	
	\textbf{Step 1.} Our first goal is to define $ \mathcal{M} $. We want a $ \delta > 0 $ such that for all $ \hat{f} \in \smooth (\mathbb{R}^m, \mathbb{R}^n) $, the inequality $ \Vert \hat{f} - f \Vert (r,A) < \delta $ implies
	\begin{align*}
		\Vert h \circ (\hat{f} - f) \Vert (r,A,p) < \frac{\epsilon}{2} && \mbox{and} && \hat{f} (A) \subseteq A'.
	\end{align*}
	By Lemma \ref{greenlemma} it is possible to choose $ \delta $ such that the first property holds. We may choose $ \delta $ such that the second property also holds, because $ \hat{f}(A) = (\hat{f}-f)(A) + f(A) \subseteq B_\delta^n (0) + f(A) $. Pick such a $ \delta $ and define
	\begin{displaymath}
		\mathcal{M} := \mathcal{N}^r (f;A,\delta).
	\end{displaymath}
	
	Observe that by the triangle inequality (Lemma \ref{triangleinequality}) there exists an $ R > 0 $ such that every $ \hat{f} \in \mathcal{M} $ satisfies $ \Vert \hat{f} \Vert (r,A) \leq R $.
	
	\textbf{Step 2.} Our second goal is to define $ \mathcal{M}' $. We want a $ \delta' > 0 $ such that for all $ \hat{h} \in \smooth (\mathbb{R}^n,E) $ and all $ \hat{f} \in \mathcal{M} $,
	\begin{align*}
		\Vert \hat{h} - h \Vert (r,A',p) < \delta' && \implies && \Vert (\hat{h} - h) \circ \hat{f} \Vert (r,A,p) < \frac{\epsilon}{2}.
	\end{align*}
	A $ \delta' $ having this property exists by Lemma \ref{greenlemma} and the observation at the end of step 1. Now define
	\begin{displaymath}
		\mathcal{M}' := \mathcal{N}^r (h;A',p,\delta').
	\end{displaymath}
	
	\textbf{Step 3.} Now we must show that $ \mathcal{M} $ and $ \mathcal{M}' $ have the desired property. Let $ \hat{f} \in \mathcal{M} $ and $ \hat{h} \in \mathcal{M}' $. By the triangle inequality,
	\begin{displaymath}
		\Vert \hat{h} \circ \hat{f} - h \circ f \Vert (r,A,p) \leq \Vert (\hat{h} - h) \circ \hat{f} \Vert (r,A,p) + \Vert h \circ (\hat{f} - f) \Vert (r,A,p) < \epsilon.
	\end{displaymath} 
	So $ \hat{h}\circ \hat{f} \in \mathcal{N} $. Thus $ \Gamma (\mathcal{M} \times \mathcal{M}') \subseteq \mathcal{N} $.
\end{proof}

A version of the preceding lemma still holds if we replace $ \mathbb{R}^m $ and $ \mathbb{R}^n $ with finite-dimensional smooth manifolds and $ E $ with an infinite-dimensional smooth manifold. This is our next result.

\begin{lemma}
	\label{illmannslemma}
	Given $ f \in \smooth (M,N) $ and $ h \in \smooth (N,X) $ and an arbitrary elementary neighborhood
	\begin{displaymath}
		\mathcal{N} := \mathcal{N}^r (h\circ f;A,(U,\phi),(W,\eta),p,\epsilon) \subseteq \vsmooth (M,X)
	\end{displaymath}
	of $ h\circ f $ there exist finitely many
	\begin{align*}
		\mathcal{M}_j := \mathcal{N}^r (f;A_j,(U,\phi),(V_j,\psi_j),\delta_j) && \mbox{and} && \mathcal{M}_j' := \mathcal{N}^r (h;A_j',(V_j,\psi),(W,\eta),p,\delta_j')
	\end{align*}
	such that
	\begin{enumerate} 
		\item $\displaystyle\Gamma \left( \bigcap_j \left( \mathcal{M}_j \times \mathcal{M}_j' \right) \right) \subseteq \mathcal{N} $,
		\item $\displaystyle \bigcup_j A_j = A $,
		\item $ f(A_j) \subseteq \interior A_j' \subseteq A_j' \subseteq V_j $ for all $ j $.
	\end{enumerate}
	
	Moreover, given any neighborhood $ Q $ of $ f(A) $ we may choose the $ V_j $ such that all $ V_j \subseteq Q $.
\end{lemma}
\begin{proof}
	Since $ f(A) $ is compact we may choose finitely many sets
	\begin{align*}
		D_j \subseteq \interior A_j' \subseteq A_j' \subseteq V_j \subseteq N && \mbox{such that} && f(A) \subseteq \bigcup_j D_j,
	\end{align*}
	where $ D_j $ and $ A_j' $ are compact, and $ V_j $ is a chart domain for a chart $ (V_j, \psi_j) $ on $ N $. Shrinking the $ V_j $ we may assume that every $ V_j \subseteq Q $. Set
	\begin{displaymath} 
		A_j := A \cap f^{-1} (D_j),
	\end{displaymath} 
	to obtain compact sets that satisfy
	\begin{align*}
		\bigcup_j A_j = A && \mbox{and} && f(A_j) \subseteq D_j \mbox{, for all $ j $.}
	\end{align*}
	
	Let
	\( \mathcal{N}_j := \bigcap_j \mathcal{N}^r (h\circ f ; A_j, (U,\phi),(W,\eta),p,\epsilon) \),
	and note that $ \mathcal{N} = \bigcap_j \mathcal{N}_j $. For each $ j $ apply Lemma \ref{realdomains} to the maps $ \psi_j \circ f \circ \phi^{-1} $ and $ \eta \circ h \circ \psi_j^{-1} $ and the elementary neighborhood
	\begin{displaymath}
		\tilde{\mathcal{N}}_j := \mathcal{N}^r (\eta \circ h \circ f \circ \phi^{-1};\phi(A_j),p,\epsilon)
	\end{displaymath}
	to obtain elementary neighborhoods
	\begin{align*}
		\tilde{\mathcal{M}}_j = \mathcal{N}^r (\psi_j \circ f \circ \phi^{-1};\phi(A_j),\delta_j) && \mbox{and} && \tilde{\mathcal{M}}_j' = \mathcal{N}^r (\eta \circ h \circ \psi_j^{-1};\psi_j (A_j'),p,\delta_j')
	\end{align*}
	such that $ \Gamma (\tilde{\mathcal{M}}_j \times \tilde{\mathcal{M}}_j') \subseteq \tilde{\mathcal{N}}_j $. 
	
	The elementary neighborhoods $ \tilde{\mathcal{M}}_j $ and $ \tilde{\mathcal{M}}_j' $ of $ \psi_j \circ f \circ \phi^{-1} $ and $ \eta \circ h \circ \psi_j^{-1} $, respectively, induce elementary neighborhoods
	\begin{align*}
		\mathcal{M}_j = \mathcal{N}^r (f;A_j,(U,\phi),(V_j,\psi_j),\delta_j) && \mbox{and} && \mathcal{M}_j' = \mathcal{N}^r (h;A_j',(V_j,\psi_j),(W,\eta),p,\delta_j')
	\end{align*}
	of $ f $ and $ h $, respectively. These neighborhoods correspond to each other in the sense that $ \hat{f} \in \mathcal{M}_j $ if and only if $ \psi_j \circ \hat{f} \circ \phi^{-1} \in \tilde{\mathcal{M}}_j $, and $ \hat{h} \in \mathcal{M}_j' $ if and only if $ \eta \circ \hat{h} \circ \psi_j^{-1} \in \tilde{\mathcal{M}}_j' $. Similarly for $ \tilde{\mathcal{N}}_j $ and $ \mathcal{N}_j $. Since $ \Gamma (\tilde{\mathcal{M}}_j \times \tilde{\mathcal{M}}_j') \subseteq \tilde{\mathcal{N}}_j $, one has by the correspondence described here that $ \Gamma (\mathcal{M}_j \times \mathcal{M}_j') \subseteq \mathcal{N}_j $. 
	
	Now just observe that
	\begin{displaymath}
		\Gamma \left( \bigcap_j \left( \mathcal{M}_j \times \mathcal{M}_j' \right) \right) \subseteq \bigcap_j \Gamma \left( \mathcal{M}_j \times \mathcal{M}_j' \right) \subseteq \bigcap_j \mathcal{N}_j = \mathcal{N}.
	\end{displaymath}
\end{proof}
\begin{lemma}
	\label{compislemma}
	Consider smooth maps $ f \in \vsmooth (M,N) $ and $ h \in \vsmooth (N,X) $ and a basic neighborhood $ \mathcal{U} = \bigcap_{i \in \Lambda} \mathcal{N}_i $, where each
	\begin{displaymath} 
	\mathcal{N}_i = \mathcal{N}^{r_i} (h\circ f;A_i,(U_i,\phi_i),(W_i,\eta_i),p_i,\epsilon_i)
	\end{displaymath}
	is an elementary neighborhood of $ h\circ f $. 
	
	If $ \lbrace f(A_i) \rbrace_{i \in \Lambda} $ is locally finite, then there exist basic neighborhoods $ \mathcal{V} $ and $ \mathcal{V}' $ of $ f $ and $ h $, respectively, such that $ \Gamma (\mathcal{V} \times \mathcal{V}') \subseteq \mathcal{U} $.
\end{lemma}
\begin{proof}
	Since $ \lbrace f(A_i) \rbrace_{i\in\Lambda} $ is locally finite, there exist compact neighborhoods $ Q_i $ of $ f(A_i) $ such that $ \lbrace Q_i \rbrace_{i \in \Lambda} $ is locally finite, by \cite[30.C.10]{cech}. Here we use our assumption that finite-dimensional manifolds are $ \sigma $-compact.
	
	For each $ i \in \Lambda $, Lemma \ref{illmannslemma} implies that there exist
	\begin{align*}
		\mathcal{W}_i :=& \bigcap_{j=1}^{n_i} \mathcal{N}^{r_i} (f;A_{i,j},(U_i,\phi_i),(V_{i,j},\psi_{i,j}),\delta_{i,j}), \\
		\mathcal{W}_i' :=& \bigcap_{j=1}^{n_i} \mathcal{N}^{r_i} (h;A_{i,j}',(V_{i,j},\psi_{i,j}),(W_i,\eta_i),p_i,\delta_{i,j}')
	\end{align*}
	such that
	\begin{enumerate}
		\item $\displaystyle \Gamma (\mathcal{W}_i \times \mathcal{W}_i') \subseteq \mathcal{N}_i $,
		\item $\displaystyle \bigcup_{j=1}^{n_i} A_{i,j} = A_i $,
		\item $ A_{i,j}' \subseteq V_{i,j} \subseteq Q_i $ for all \( j \).
	\end{enumerate}
	Then $ \lbrace A_{i,j} \rbrace_{i,j} $ is locally finite by (2) and since $ \lbrace A_i \rbrace_i $ is locally finite, and $ \lbrace A_{i,j}' \rbrace_{i,j} $ is locally finite by (3) and since $ \lbrace Q_i \rbrace_i $ is locally finite.	Hence
	\begin{align*}
		\mathcal{V} := \bigcap_{i\in\Lambda} \mathcal{W}_i && \mbox{and} && \mathcal{V}' := \bigcap_{i\in\Lambda} \mathcal{W}_i'
	\end{align*}
	are basic neighborhoods of $ f $ and $ h $, respectively, such that
	\begin{displaymath}
		\Gamma \left( \mathcal{V} \times \mathcal{V}' \right) = \Gamma \left( \bigcap_{i\in\Lambda} \left( \mathcal{W}_i \times \mathcal{W}_i' \right) \right) \subseteq \bigcap_{i\in\Lambda} \Gamma (\mathcal{W}_i \times \mathcal{W}_i') \subseteq \bigcap_{i\in\Lambda} \mathcal{N}_i = \mathcal{U}.
	\end{displaymath}
\end{proof}
\begin{theorem}
	\label{compiscont}
	Let $M$ and $N$ be finite-dimensional smooth manifolds and let $ X $ be a smooth manifold modeled on a locally convex vector space $E$. Then the composition map
	\begin{displaymath}
	\Gamma \colon \Prop_{\textup{vS}} (M,N) \times \vsmooth (N,X) \to \vsmooth (M,X)
	\end{displaymath}
	sending $ (f,h) $ to $ h \circ f $ is continuous.
\end{theorem}
\begin{proof}
	It suffices to show that given maps $ f \in \Prop_{\textup{vS}} (M,N) $ and $ h \in \vsmooth (N,E) $ and a basic neighborhood
	\begin{displaymath}
		\mathcal{U} = \bigcap_{i\in\Lambda} \mathcal{N}^{r_i} (h\circ f;A_i,(U_i,\phi_i),(V_i,\psi_i),p_i,\epsilon_i)
	\end{displaymath}
	of $ h\circ f $ in $ \vsmooth (M,X) $, there exist basic neighborhoods $ \mathcal{V} $ and $ \mathcal{V}' $ of $ f $ and $ h $, respectively, such that $ \Gamma (\mathcal{V} \times \mathcal{V}') \subseteq \mathcal{U} $.
	
	So suppose that we are given $ f , h $ and $ \mathcal{U} $ as above. Then $ \lbrace f(A_i) \rbrace_{i \in \Lambda} $ is locally finite since $ f $ is proper, by \cite[Lemma 3.10.11]{engelking}. Thus we may apply Lemma \ref{compislemma} to obtain the desired neighborhoods $ \mathcal{V} $ and $ \mathcal{V}' $.
\end{proof}
Unfortunately, precomposition is not continuous in general as an examination of Example \ref{counterexample} reveals. However, precomposition by a proper map is continuous. 
\begin{proposition}
	\label{precomposition}
	Let $ f \in \Prop_{\textup{vS}} (M,N)$. Then the following map is continuous 
	\begin{align*}
		f^* \colon \vsmooth (N,X) &\to \vsmooth (M,X), \quad		h \mapsto h \circ f 
	\end{align*}
\end{proposition}
\begin{proof}
	The map $ \iota_f \colon \vsmooth (N,X) \to \Prop (M,N) \times \vsmooth (N,X) $ given by $ \iota_f (h) = (f,h) $ is continuous. Hence $ f^* $, which is the composition 
	\begin{displaymath}
		\vsmooth (N,X) \xrightarrow{\iota_f} \Prop_{\textup{vS}} (M,N) \times \vsmooth (N,X) \xrightarrow{\Gamma} \vsmooth (M,X),
	\end{displaymath}
	is also continuous. 
\end{proof}

We will now prove that postcomposition is always continuous. 
This result is needed even for postcomposition by a map $f\colon X \rightarrow Y$ between infinite-dimensional manifolds.
Thus the next proposition can not readily be deduced from Theorem \ref{compiscont} (or the other results in this section)\footnote{For the finite-dimensional case, a proof along these lines can be found in \cite{illman}.}. 
Instead we have to take a detour using results on the (coarser) compact open $C^\infty$-topology (cf.\ to the proof of Lemma \ref{lem: atlaschoice}). 

\begin{proposition}\label{postcomposition}
 Let $f \colon X \rightarrow Y$ be smooth, where $Y$ is a (possibly infinite-dimensional) manifold. 
 Then for any finite-dimensional manifold $M$ the following map is continuous 
      \begin{displaymath}
       f_* \colon \vsmooth (M,X) \rightarrow \vsmooth (M,Y) , \quad h \mapsto f\circ h
      \end{displaymath}
\end{proposition}

\begin{proof}
 To see that  $f_*$ is continuous, we proceed in several steps.
 
 \textbf{Step 1} \emph{Special elementary neighborhoods.} 
 Consider first an arbitrary elementary neighborhood $\mathcal{N} = \mathcal{N}^{r} (f\circ h; A,(U, \phi),(V,\psi),q, \epsilon)$ in $\vsmooth (M,Y)$.
 Since $h(A)$ is compact, there are finitely many manifold charts $(W_i,\kappa_i)$ of $X$ with $h(A) \subseteq \bigcup_{i} W_i$.
 Now the open sets $h^{-1} (W_i)$ cover $A$ and thus there are finitely many compact sets $K_j$ such that $A=\bigcup_{j} K_j$ and $h(K_j) \subseteq W_{i_j}$.
 Thus we replace $A$ by the finitely many compact sets. Note that this will ensure that the families of compact sets considered later remain locally finite.
 To shorten the notation, assume without loss of generality that there is a manifold chart $(W,\kappa)$ of $X$ such that $h(A) \subseteq W$ and $f(W) \subseteq V$.
 In particular, we can thus consider the mapping $f_{\kappa}^\psi := \psi \circ f \circ \kappa^{-1} \colon \kappa (W) \rightarrow \psi (V)$
 
 \textbf{Step 2} \emph{The preimage of a special elementary neighborhood of $f\circ h$ is a neighborhood of $h$.} We work locally in charts.
 Let $Y$ be modeled on the locally convex space $F$ and $X$ be modeled on the locally convex space $E$.
 Recall that the compact open $C^\infty$-topology (see \cite[Definition I.5.1]{neeb}) controls the derivatives of functions on compact sets.
 Moreover, the elementary neighborhoods of the very strong topology form a subbase of the compact open $C^\infty$-topology.
 We denote by $C^\infty (M,F)_{\text{co}}$ the vector space of smooth functions with the compact open $C^\infty$-topology. 
 
 Choose a compact neighborhood $C \subseteq U$ of $A$ such that $h(C) \subseteq W$ (this entails $f\circ h(C)\subseteq V$). 
 We endow the subset $\lfloor C,\kappa(W)\rfloor := \{g\in C^\infty (M,E) \mid g(C) \subseteq \kappa(W)\}$ with the subspace topology induced by the compact open $C^\infty$-topology.
 As $\kappa (W) \subseteq E$ is open, we note that $\lfloor C,\kappa(W)\rfloor$ is open in $C^\infty (M,E)_{\text{co}}$.
 Now \cite[Proposition 4.23 (a)]{glockomega} shows that 
 \begin{displaymath}
  (f_{\kappa}^\psi)_* \colon \lfloor C,\kappa(W) \rfloor \rightarrow C^\infty (\interior C,F)_{\text{co}},\quad  h \mapsto (\psi\circ f\circ \kappa^{-1}) \circ h
 \end{displaymath}
 is continuous. 
 Moreover, $\mathcal{N}_{loc} := \mathcal{N}^{r} (\psi \circ f\circ h|_{\interior C}; A,(U\cap \interior C, \phi),(F,\id_F),q, \epsilon)$ is open in $C^\infty (\interior C,F)_{\text{co}}$.
 Further, $f_{\kappa}^\psi \circ \kappa \circ h|_{\interior C} = \psi \circ f \circ h|_{\interior C}$. 
 Observe that thus $\kappa \circ h \in  ((f_{\kappa}^\psi)_*)^{-1} (\mathcal{N}_{loc})$. 
 As the elementary neighborhoods form a subbase of the compact open $C^\infty$-topology,  Lemma \ref{lem: atlaschoice} together with continuity of $(f_{\kappa}^\psi)_*$ yields
 \begin{equation}\label{eq: locpush}
   \lfloor C, \kappa (W)\rfloor \cap \bigcap_{k=1}^N \mathcal{N}^{r} (\kappa \circ h; A_k,(U_k, \phi_k),(E,\id_E),p_k, \epsilon_k)  \subseteq  ((f_{\kappa}^\psi)_*)^{-1} (\mathcal{N}_{loc}).
 \end{equation}
 Recall from the proof of \cite[Proposition 4.23 (a)]{glockomega} that the compact sets $A_k$ are contained by construction in $\interior C$. 
 Thus one easily deduces from \eqref{eq: locpush} that 
 \begin{displaymath}
  \mathcal{N}^{0} (h; C,(U, \phi),(W,\kappa),p_1, 1) \cap \bigcap_{k=1}^N \mathcal{N}^{r} (h; A_k,(U_k, \phi_k),(W,\kappa),p_k, \epsilon_k) \subseteq (f_*)^{-1} (\mathcal{N})
 \end{displaymath}
 Summing up, we see that $(f_*)^{-1} (\mathcal{N})$ is a neighborhood of $h$. 
 Further, this \emph{finite} family of neighborhoods controls the behavior of mappings only on a pre chosen compact set $C$ (which depends of course on $h$). 
 
 \textbf{Step 3} \emph{Preimages of basic neighborhoods are open}  
 Let $\mathcal{M} = \bigcap_{i \in \mathbb{N}} \mathcal{N}_i$ be a basic neighborhood of $f\circ h \in C^\infty (M,Y)$ with $\{A_k\}_{k \in \mathbb{N}}$ its the underlying compact family.
 We will prove that for arbitrary $g \in (f_*)^{-1} (\mathcal{M})$ the preimage is a neighborhood of $g$.
 Choose with Proposition \ref{basicnbhsisbasis} a basic neighborhood of $f\circ g$ which is contained in $\mathcal{M}$. 
 Replacing $\mathcal{M}$ with this basic neighborhood, it suffices thus to consider the case $g=h$. 
 Splitting each $A_k$ as in Step 1 we may assume without loss of generality that each $\mathcal{N}_i$ is of the form considered in Step 2.
 Use \cite[30.C.10]{cech} to construct for every $A_k$ a compact neighborhood $C_k$ such that $\{C_k\}_{k \in \mathbb{N}}$ is locally finite.
 Now we proceed for every elementary neighborhood $\mathcal{N}_k$ as above (replace $C$ in Step 2 by $C_k$ and shrink $C_k$ if necessary!). 
 Since the family $\{C_k\}_{k \in \mathbb{N}}$ is locally finite, we thus end up with a basic neighborhood $\mathcal{M}_h$ around $h$ which mapped by $f_*$ to $\mathcal{M}$.
 We conclude that $(f_*)^{-1} (M)$ is a neighborhood of $h$, whence of every of its elements.
 Hence preimages of basic neighborhoods under $f_*$ are open in $\vsmooth (M,X)$, whence $f_*$ is continuous.
\end{proof}

As an application, we can now identify (as topological spaces) spaces of maps into a product with products of spaces of mappings to the factors.

\begin{theorem}
	\label{product theorem}
	Let $ M $ be a finite-dimensional manifold, and let $ X_1 $ and $ X_2 $ be smooth manifolds modeled on locally convex vector spaces $ E_1 $ and $E_2$, respectively. Then 
	\begin{align*}
		\iota : \vsmooth (M, X_1 \times X_2) \to \vsmooth (M, X_1) \times \vsmooth (M, X_2), \quad f \mapsto (\pr_1 \circ f, \pr_2 \circ f)
	\end{align*}
	is a homeomorphism, where for \( i \in \lbrace 1,2 \rbrace \) \( \pr_i \colon X_1 \times X_2 \to X_i \) is the canonical projection.
\end{theorem}
\begin{proof}
	Clearly $ \iota $ is a bijection, and it is continuous by Proposition \ref{postcomposition}. We will prove that $ \iota^{-1} $ is continuous, i.e. that $ \iota $ is open.
	By \eqref{locally convex product} in Proposition \ref{locally convex prop}, the set 
		\begin{displaymath}
			\mathcal{P} := \lbrace p \circ \pr_i : p \mbox{ is a continuous seminorm on } E_i \rbrace
		\end{displaymath} 
	is a generating family of seminorms on $ E_1 \times E_2 $.
	Consider a basic neighborhood $ \mathcal{U} = \bigcap_{i\in\Lambda} \mathcal{N}_i $ of $ f \in \vsmooth (M,X_1 \times X_2) $, where each
	$
		\mathcal{N}_i = \mathcal{N}^{r_i} (f;A_i,(U_i,\phi_i),(V_i,\psi_i),p_i,\epsilon_i).
	$
	By Proposition \ref{generatingfamilies} we may assume that each $ p_i \in \mathcal{P} $. Take an arbitrary $ i \in \Lambda $. If $ p_i = p\circ \pr_1 $ for some continuous seminorm $ p $ on $ E_1 $, let
	\begin{align*}
		\mathcal{M}_i = \mathcal{N}^{r_i} (\pr_1 \circ f; A_i,(U_i,\phi_i),(V_i,\psi_i),p,\epsilon_i) && \mbox{and} && \mathcal{M}_i' = \vsmooth (M,X_2).
	\end{align*}
	If $ p_i = q\circ \pr_2 $ for a continuous seminorm $ q $ on $ E_2 $, reverse the roles of $ \mathcal{M}_i $ and $ \mathcal{M}_i' $.
	
	Now suppose without loss of generalization that $ p_i = p \circ \pr_1 $ for some continuous seminorm $ p $ on $ E_1 $. For $ g \colon \mathbb{R}^m \supseteq \phi_i (A_i) \to \psi_i (V_i) \subseteq E_1 \times E_2 $, one has
	\begin{displaymath}
	\dd^{(k)} g = \dd^{(k)} (\pr_1 \circ g , \pr_2 \circ g) = ( \dd^{(k)} \pr_1 \circ g, \dd^{(k)} \pr_2 \circ g ),
	\end{displaymath}
	so the condition
	\( p_i \left( \dd^{(k)}g (a;\alpha) \right) < \epsilon_i \)
	is equivalent to
	\( p \left( \dd^{(k)} (\pr_1 \circ g) (a;\alpha) \right) < \epsilon_i \).
	Hence $ \mathcal{M}_i \times \mathcal{M}_i' = \iota (\mathcal{N}_i) $. Since $ \iota $ is bijective one has
	\begin{displaymath}
		\iota (\mathcal{U}) = \bigcap_{i\in\Lambda} \iota (\mathcal{N}_i) = \bigcap_{i\in\Lambda} \mathcal{M}_i \times \bigcap_{i\in\Lambda} \mathcal{M}_i'.
	\end{displaymath}
	So $ \iota $ is open, and a homeomorphism.
\end{proof}

\begin{corollary}
	\label{product corollary}
	If $ Q $ is a compact smooth manifold, then following map is continuous
	\begin{align*}
		\chi \colon \vsmooth (M,X) \to \vsmooth (Q \times M, Q \times X), \quad f \mapsto \id \times f.
	\end{align*}
\end{corollary}
\begin{proof}
	By Theorem \ref{product theorem} it suffices to show that the maps
	\begin{align*}
		\chi_1 \colon \vsmooth (M,X) &\to \vsmooth (Q\times M, Q) && &\mbox{and} && \chi_2 \colon \vsmooth(M,X) &\to \vsmooth (Q \times M, X) \\
		f &\mapsto \pr_1 \circ (\id \times f) && & && f &\mapsto \pr_2 \circ (\id \times f)
	\end{align*}
	are continuous. For $ (q,m) \in Q \times M $, one has
	\begin{align*}
		\chi_1(f)(q,m) &= \pr_1 \circ (\id \times f) (q,m) = \pr_1 (q,f(m)) = q = \pr_1 (q,m), \\ 
		\chi_2(f)(q,m) &= \pr_2 \circ (\id \times f) (q,m) = \pr_2 (q,f(m)) = f(m) \\ &= f \circ \pr_2 (q,m) = \pr_2^* (f) (q,m).
	\end{align*}
	The map \( \chi_1 \) is constant in $ f $, hence continuous, and the map \( \chi_2 = \pr_2^* \). Since $ Q $ is compact, $ \pr_2 $ is proper, so Proposition \ref{precomposition} implies that $ \chi_2 = \pr_2^* $ is also continuous.
\end{proof}

\section{The fine very strong topology}

In the end, we would like a structure on \( \smooth (M,X) \) as a locally convex manifold, where \( M \) is a finite-dimensional smooth manifolds and \( X \) is a manifold modeled on a locally convex vector space \( E \), but for this purpose the very strong topology is not fine enough. A first step in the direction of making \( \smooth (M,X) \) into a locally convex manifold would be having a similar structure on \( \smooth (M,E) \). One might hope that \( \vsmooth (M,E) \) itself with the vector space structure induced by pointwise operations would be a locally convex vector space. But as Corollary \ref{supportcorollary} points out, this is not the case when $ E $ is a (non-trivial) locally convex vector space and $ M $ is a non-compact manifold. However, we will see in the next section that the subspace of \( \vsmooth (M,E) \) consisting of maps with compact support, denoted \( C_{\text{vS,c}}^\infty (M,E) \), is a locally convex vector space. Following \cite{michor}, we refine the topology on \( \vsmooth (M,E) \) to obtain a structure on \( \smooth (M,E) \) as a smooth manifold modeled on \( C_{\text{vS,c}}^\infty (M,E) \). The resulting topology on \( \smooth (M,E) \), or more generally \( \smooth (M,X) \), is called the \emph{fine very strong} topology on $ \smooth (M,X) $. The space \( \smooth (M,X) \) equipped with the fine very strong topology is denoted $ \fsmooth (M,X) $. 

Fortunately, the results of the previous sections are easily extended to hold in the fine very strong topology. This is done in Proposition \ref{fsmoothprop}.

It is a folklore fact (Proposition \ref{prop: topologiescoincide}) that in the finite-dimensional case, the very strong topology is equivalent to the $ \mathcal{D} $-topology as described in \cite[36]{michor}.\footnote{The $\mathcal{D}$-topology was defined using jet bundles (also reviewed in Appendix \ref{folklore}). Our treatment of the topology has the advantage that only elementary arguments are needed. Further, only our approach generalizes to arbitrary locally convex target manifolds.}  Consequently, the fine very strong topology is equivalent to the $ \mathcal{FD} $-topology defined in \cite[40]{michor}.

\begin{proposition}
	\label{supportprop}
	Let $ M $ be a finite-dimensional smooth manifold and $ E $ be a locally convex vector space. Consider a sequence $ \lbrace f_n \rbrace_{n\in \mathbb{N}} \subseteq \vsmooth (M,E) $ which converges in the very strong topology towards \( f \in \smooth (M,E) \). 
	Then there exist a compact $ K \subseteq M $ and an $ N \in \mathbb{N} $ such that for all $ n \geq N $ we have $$ \osupp{f}{f_n} :=  \lbrace y \in M : f_n (y) \neq f(y) \rbrace \subseteq K. $$
\end{proposition}
\begin{proof}
	For $ f \in \vsmooth (M,X) $, we will show that $ f $ cannot be a limit of $ \lbrace f_n \rbrace $ if for all compact $ K \subseteq M $ and all $ N \in \mathbb{N} $ there exists $ n \geq N $ such that $ \osupp{f}{f_n} \nsubseteq K $.

	Let \( \lbrace A_n \rbrace_{n \in \mathbb{N}} \) be a locally finite exhaustion of \( M \) by compact sets (exists by Lemma \ref{dugundjifact} since \( M \) is \( \sigma \)-compact), and for \( n \in \mathbb{N} \) set \( K_n = \bigcup_{i=1}^n A_i \).

	Construct a basic neighborhood of \( f \) recursively, using the following procedure. Let \( n_0 = 1 \), \( m_0 = 1 \). For \( i \in \mathbb{N} \), choose \( n_i > n_{i-1} \) such that \( \osupp{f}{f_{n_i}} \nsubseteq K_{m_{i-1}} \). By construction there exists \( m_i > m_{i-1} \) such that \( \osupp{f}{f_{n_i}} \cap \left( M \setminus K_{m_{i-1}} \right) \cap A_{m_i} \neq \emptyset \). Take any \( x \) in this nonempty set. Since \( f (x) \neq f_{n_i} (x) \), there exists a continuous seminorm \( p_i \) on \( E \) such that \( 2 \epsilon_i := p_i (f_{n_i} (x) - f(x)) > 0 \), and then
	\[
		f_{n_i} \notin \mathcal{N}_i := \mathcal{N}^0 \left( f; A_{m_i}, p_i, \epsilon_i \right).
	\]
	Now \( \mathcal{U} := \bigcap_{i \in \mathbb{N}} \mathcal{N}_i \) is a basic neighborhood of \( f \) such that for all \( N \in \mathbb{N} \) there exists \( n \geq N \) such that \( f_n \notin \mathcal{U} \). So the sequence \( \lbrace f_n \rbrace_{n \in \mathbb{N}} \) does not converge to \( f \).
\end{proof}
\begin{remark}
	One can easily prove the proposition above for \( E \) a locally convex manifold rather than a locally convex vector space, by ``hacking'' the compact sets \( A_i \) in the proof into smaller compact sets that are contained in charts.
\end{remark}
\begin{corollary}
	\label{supportcorollary}
	Let $ M $ be a finite-dimensional non-compact manifold and $ E \neq \lbrace 0 \rbrace $ a locally convex vector space. Then $ \vsmooth (M,E) $ with the vector space structure induced by pointwise operations is not a topological vector space.
\end{corollary}
\begin{proof}
	Let $ f \in \vsmooth (M,E) $ be a non-zero constant map. Then Proposition \ref{supportprop} shows that $ \lim_{\lambda \to 0} (\lambda f) \neq 0 = (\lim_{\lambda \to 0} \lambda) f $, hence scalar multiplication is not continuous.
\end{proof}
\begin{remark}
	\label{supportcorollaryremark}
	Although $ \vsmooth (M,E) $ is not a topological vector space, it is a topological group under pointwise addition by Lemma \ref{continuous addition}. And $ \vsmooth (M,\mathbb{R}) $ is a topological ring under the pointwise operations induced by addition and multiplication in $ \mathbb{R} $.
\end{remark}
\begin{definition}[The fine very strong topology]
	Define an equivalence relation $ \sim $ on $ \smooth (M,X) $ by declaring that $ f \sim g $ whenever $$ \csupp{f}{g} := \overline{\lbrace y \in M : f(y) \neq g(y) \rbrace} $$ is compact. Now refine the very strong topology on $ \smooth (M,X) $ by demanding that the equivalence classes are open in $ \smooth (M,X) $. In other words, equip $ \smooth (M,X) $ with the topology generated by the very strong topology and the equivalence classes. This is the \emph{fine very strong topology} on $ \smooth (M,X) $. We write $ \fsmooth (M,X) $ for $ \smooth (M,X) $ equipped with the fine very strong topology.
\end{definition}
\begin{remark}
	\label{fvs remark}
	Here is another way to look at the fine very strong topology. Start with $ \vsmooth (M,X) $ and equip the equivalence classes $ [f] $ with the subspace topology. Then $$ \fsmooth (M,X) = \bigsqcup_{[f] \in \smooth (M,X)/\sim} [f] $$ as topological spaces.
	Taking the family of all sets of the form $ \mathcal{U} \cap [f] $, where $ \mathcal{U} $ runs through the basic neighborhoods in $ \smooth (M,X) $ and $ [f] $ runs through the equivalence classes, yields a basis for the fine very strong topology on $ \smooth (M,X) $.
\end{remark}
\begin{remark}
	If $ f \in \smooth (M,X) $ is a proper map and $ f \sim \hat{f} $, then $ \hat{f} $ is also proper. Indeed, if $ K \subseteq X $ is compact, then $ \hat{f}^{-1} (K) \subseteq f^{-1}(K) \cup \csupp{f}{\hat{f}} $. Since closed subspaces of compact spaces are compact, $ \hat{f}^{-1} (K) $ is compact.
\end{remark}
We would obviously like the results of the previous sections to remain true in the fine very strong topology. Fortunately, it is easy to extend the results to this case using the following lemma.
\begin{lemma}
	\label{corollarylemma}
	Let $ T $ be a topological space, and $ \zeta \colon T \to \smooth (M,X) $ a function. If $ \zeta $ is continuous as a map to $ \vsmooth (M,X) $ and $ \zeta^{-1} ([f]) \subseteq T $ is open for all equivalence classes $ [f] \subseteq \smooth (M,X) $, then $ \zeta $ is continuous as a map to $ \fsmooth (M,X) $.
\end{lemma}
\begin{proof}
	The map $ \zeta $ is continuous if preimages of basis elements are open. Basis elements for $ \fsmooth (M,X) $ are of the form $ \mathcal{U} \cap [f] $ for some basic neighborhood $ \mathcal{U} $ and some equivalence class $ [f] $, and $ \zeta^{-1}(\mathcal{U}\cap [f]) = \zeta^{-1}(\mathcal{U}) \cap \zeta^{-1} ([f]) $.
\end{proof}
\begin{proposition}
	\label{fsmoothprop}
	Theorem \ref{compiscont}, Proposition \ref{precomposition}, Proposition \ref{postcomposition}, Theorem \ref{product theorem}, and Corollary \ref{product corollary} still hold if we in every case replace the very strong topology with the fine very strong topology.
	
	In the cases that we consider $ \Prop_{\textup{vS}} (M,N) $, replace this with $ \Prop_{\mbox{fS}} (M,N) $, by which is meant the subset $ \Prop (M,N) \subseteq \fsmooth (M,N) $ equipped with the subspace topology.
\end{proposition}
\begin{proof}
	The proof is case by case. In all cases except for the generalization of Theorem \ref{product theorem} and its corollary, it suffices by \ref{corollarylemma} to check that preimages of equivalence classes are open. Unless otherwise stated, letters (such as $ f $ or $ N $) are always assumed to have the same role here as in the statement of the corresponding result.

	\textit{Theorem \ref{compiscont} (the full composition map is continuous).} Suppose that $ f \sim \hat{f} $ and $ h \sim \hat{h} $. We have $ \csupp{h \circ f}{ \hat{h} \circ \hat{f} } \subseteq \csupp{f}{\hat{f}} \cup f^{-1}\left( \csupp{h}{\hat{h}} \right) $. The right hand side  is compact since $ f $ is proper, so $ \csupp{h \circ f}{ \hat{h} \circ \hat{f} } $ is a closed subset of a compact space, hence compact. By definition this means that $ h \circ f \sim \hat{h} \circ \hat{f} $. 
		
	Consider an equivalence class $ [g] \subseteq \fsmooth (M,X) $. By what we just observed, if $ h \circ f \sim g $ and $ \hat{f} \sim f $ and $ \hat{h} \sim h $, then $ \hat{h} \circ \hat{f} \sim g $. Hence $$ \Gamma^{-1} ([g]) = \bigcup_{h \circ f \sim g} [f] \times [h], $$ which is open.
	
	\textit{Proposition \ref{precomposition} (precomposition is continuous).} If $ h \sim \hat{h} $, then $ h \circ f \sim \hat{h} \circ f $ by the same argument as before. So for any equivalence class $ [g] \subseteq \smooth (M,X) $, we have $$ (f^*)^{-1} ([g]) = \bigcup_{h \circ f \sim g} [h]. $$
	
	\textit{Proposition \ref{postcomposition} (postcomposition is continuous).} If $ h, \hat{h} \in \smooth (M,X) $, then it is easy to see that $ \csupp{f \circ h}{f \circ \hat{h}} \subseteq \csupp{h}{\hat{h}} $. So if $ h \sim \hat{h} $, then $ f \circ h \sim f \circ \hat{h} $, since closed subsets of compact spaces are compact. It follows that for any equivalence class $ [g] \subseteq \smooth (M,X) $, we have $ (f_*)^{-1}([g]) = \bigcup_{f \circ h \sim g} [h]. $ 
	
	\textit{Theorem \ref{product theorem} (the product theorem).} For the same reasons as in the proof of the very strong version of the theorem, $ \iota $ is clearly a bijective continuous map. So by Lemma \ref{corollarylemma} it suffices to show that images of equivalence classes are open. 
	
	Observe that for $ f, \hat{f} \in \smooth (M, X_1 \times X_2) $, we have $ \csupp{f}{\hat{f}} = \csupp{\pr_1 \circ f}{\pr_1 \circ \hat{f}} \cup \csupp{\pr_2 \circ f}{\pr_2 \circ \hat{f}} $. Hence $ f \sim \hat{f} $ if and only if $ \pr_1 \circ f \sim \pr_1 \circ \hat{f} $ and $ \pr_2 \circ f \sim \pr_2 \circ \hat{f} $. Another way of stating this fact is $ \iota ([f]) = [\pr_1 \circ f] \times [\pr_2 \circ f] $ for all $ f \in \smooth (M,X_1 \times X_2) $.
	
	\textit{Corollary \ref{product corollary}.} Same proof as in the very strong case.
\end{proof}

\section{The manifold structure on smooth vector valued functions}\label{sect: smmfd}
Throughout this section, $ M $ is a finite-dimensional manifold, $ E $ is a locally convex vector space, and \( X \) is a locally convex manifold.

Recall from Corollary \ref{supportcorollary} that \( \vsmooth (M,E) \) with pointwise operations is not a locally convex vector space, in fact it is not even a topological vector space. Neither is \( \fsmooth (M,E) \). However, we will in this section make \( \fsmooth (M,E) \) into a locally convex manifold. This is a first step towards making \( \fsmooth (M,X) \) into a locally convex manifold (but we will not do this). The modeling space for \( \fsmooth (M,E) \) as a locally convex manifold is \( C_{\text{vS,c}}^\infty (M,E) \), defined below.

\begin{definition}
	We define $ C_{\text{vS,c}}^\infty (M,E) $ to be the subspace of $ \vsmooth (M,E) $ consisting of the functions with compact support, i.e. 
	$$ C_{\text{vS,c}}^\infty (M,E) = \lbrace f \in \vsmooth (M,E) : \mbox{$\csupp{f}{0}$ is compact} \rbrace $$ equipped with the subspace topology from $ \vsmooth (M,E) $.
	Note that $ C_{\text{vS,c}}^\infty (M,E) = [0] $ in $ \fsmooth (M,E) $.
\end{definition}

As a first step towards proving that \( C_{\text{vS,c}}^\infty (M,E) \) with pointwise operations is a locally convex vector space, we show that \( \smooth (M,E) \) with pointwise addition is a topological group in the very strong and fine very strong topologies.

\begin{lemma}
	\label{continuous addition}
	Addition
	\begin{align*}
		\Sigma \colon \smooth (M,E) \times \smooth (M,E) &\to \smooth (M,E), \quad (f,g) &\mapsto f+h = \left[ m \mapsto f(m) + h(m) \right]
	\end{align*}
	is continuous when $ \smooth (M,E) $ is equipped with the very strong topology or fine very strong topology.
\end{lemma}
\begin{proof}
	We prove the assertion only for the very strong topology as the proof carries over verbatim to the fine very strong topology.
	By Theorem \ref{product theorem} there is a canonical homeomorphism $ \iota \colon \vsmooth (M,E) \times \vsmooth (M,E) \cong \vsmooth (M,E \times E ) $. Since addition $ S \colon E \times E \to E $ in $ E $ is smooth, induced postcomposition $ S_* \colon \vsmooth (M,E\times E) \to \vsmooth (M,E) $ is continuous.
	Hence $ \Sigma = S_* \circ \iota $ is continuous.
\end{proof}

Once we have established the following proposition, it will be easy to make \( \fsmooth (M,E) \) into a locally convex manifold modeled on \( C_{\text{vS,c}}^\infty (M,E) \). The hard work lies here.
\begin{proposition}
	\label{locally convex vs prop}
	The topological space $ C_{\text{vS,c}}^\infty (M,E) $ with vector space structure induced by pointwise operations in $ E $ is a locally convex vector space.
\end{proposition}
\begin{proof}
	In Lemma \ref{continuous addition} we showed that addition is continuous, and the topological space $ C_{\text{vS,c}}^\infty (M,E) $ is Hausdorff since the compact open \( \smooth \)-topology on \( \smooth(M,E) \) is Hausdorff and the very strong topology is finer than the compact open \( \smooth \)-topology. It is therefore only necessary to check that scalar multiplication is continuous in order to conclude that $ C_{\text{vS,c}}^\infty (M,E) $ is a topological vector space. Finally, we must verify that this topological vector space is locally convex.
	
	\textit{Scalar multiplication $\mu \colon \mathbb{R} \times C_{\text{vS,c}}^\infty (M,E) \to C_{\text{vS,c}}^\infty (M,E),\  
		(\lambda, f) \mapsto \lambda f $ is continuous.} 
	Let $ (\lambda,f) \in \mathbb{R} \times C_{\text{vS,c}}^\infty (M,E) $, and consider a basic neighborhood $ \mathcal{V} = \bigcap_{i \in \Lambda} \mathcal{N}_i $ of $ \lambda f $, where each $\mathcal{N}_i = \mathcal{N}^{r_i} (\lambda f; A_i, (U_i,\phi_i),p_i,\epsilon_i) $ is an elementary neighborhood of $ \lambda f $. We will show that there exists open sets $ I \subseteq \RR $ and $ \mathcal{U} \subseteq \vsmooth (M,E) $ such that $ \mu (I \times \mathcal{U}) \subseteq \mathcal{V} $. 
	
	Since $ \csupp{f}{0} $ is compact, only finitely many $ A_i $ intersect $ \csupp{f}{0} $, say only for $ i = i_1, \dots , i_n $. Define $ \epsilon := \min (\epsilon_{i_1},\dots,\epsilon_{i_n}) $. 
	
	Set $m_1 := \max\left\{\sup_{1 \leq j \leq n} \Vert f \circ \phi_{i_j}^{-1} \Vert (r_{i_j}, \phi_{i_j}(A_{i_j}),p_{i_j}), 1\right\}$ and 
	\[ I := B_{\frac{\epsilon}{2 m_1 }}^1 (\lambda) = \left] \lambda - \frac{\epsilon}{2 m_1}, \lambda + \frac{\epsilon}{2 m_1} \right[ . \] 
	
	Define $ m_2 := \sup \lbrace | t | : t \in I \rbrace $, and set $ \mathcal{U} := \bigcap_{i \in \Lambda} \mathcal{N}^{r_i} \left( f;A_i,(U_i,\phi_i),p_i,\frac{\epsilon_i}{2 m_2} \right). $
	Suppose $ (\lambda',f') \in I \times \mathcal{U} $. 
	For all $ i \in \Lambda $, $ x \in A_i $, $ 1 \leq k \leq r_i $, and $ \alpha \in \lbrace e_1, \dots , e_{\dim M} \rbrace^k $, we have
	\begin{align*}
		& p_i \left( \dd^{(k)} (\lambda' f' \circ \phi_i^{-1} - \lambda f \circ \phi_i^{-1})(\phi_i(x);\alpha) \right) \\
		\leq& | \lambda' | p_i \left( \dd^{(k)} (f' \circ \phi_i^{-1} - f \circ \phi_i^{-1})(\phi_i(x);\alpha) \right) + | \lambda' - \lambda | p_i \left( \dd^{(k)} (f \circ \phi_i^{-1})(\phi_i(x);\alpha) \right) \\
		<& \frac{\epsilon_i}{2} + \frac{\epsilon}{2 m_1} p_i \left( \dd^{(k)} (f \circ \phi_i^{-1})(\phi_i(x);\alpha) \right) =: C.
	\end{align*}
	If $ i \notin \lbrace i_1,\dots,i_n \rbrace $, then $ p_i \left( \dd^{(k)} (f \circ \phi_i^{-1})(\phi_i(x);\alpha) \right) = 0 $, in which case $ C \leq \epsilon_i $. And if $ i \in \lbrace i_1, \dots,i_n \rbrace $, then $ \epsilon \leq \epsilon_i $ and $ p_i \left( \dd^{(k)} (f \circ \phi_i^{-1})(\phi_i(x);\alpha) \right) \leq m_1 $, in which case we still have $ C \leq \epsilon_i $. Hence $ \lambda' f' \in \mathcal{V} $, and $ \mu (I \times \mathcal{U}) \subseteq \mathcal{V} $. Consequently, $\mu$ is continuous.
	
	\textit{The space is locally convex.} We have now established that $ C_{\mbox{vS,c}}^\infty (M,E) $ is a topological vector space. It remains to see that this topological vector space is locally convex. For $ r \in \mathbb{N}_0 $, $ (U,\phi) $ a chart on $ M $, $ A \subseteq U $ compact, and $ p $ a continuous seminorm on $ E $, define
	\begin{align*}
		\Vert\cdot \Vert(r,A,(U,\phi),p) \colon C_{\mbox{vS,c}}^\infty (M,E) \to [0,\infty), \quad f \mapsto \Vert f \circ \phi^{-1} \Vert (r,\phi_i (A_i),p) 
	\end{align*}
	This is a seminorm on $ C_{\mbox{vS,c}}^\infty (M,E) $. Consider a family $ \lbrace \Vert\cdot \Vert (r_i,A_i,(U_i,\phi_i),p_i) \rbrace_{i\in \Lambda} $ of such seminorms, where $ \lbrace A_i \rbrace_{i \in \Lambda} $ is locally finite. For some family $ \lbrace \epsilon_i \rbrace_{i \in \Lambda} $ define $ q \colon C_{\mbox{vS,c}}^\infty (M,E) \to [0,\infty) $ by \[ q(f) = \sup_{i\in\Lambda} \epsilon_i \Vert f \Vert (r_i,A_i,(U_i,\phi_i),p_i). \] 
	Every $ f \in C_{\mbox{vS,c}}^\infty (M,E) $ has compact support, so $ \supp(f,0) $ intersects only finitely many of the $ A_i $, from which it follows that $ \Vert f \Vert (r_i,A_i,(U_i,\phi_i),p_i) \neq 0 $ for only finitely many $ i \in \Lambda $. Hence $ q(f) < \infty $, so $ q $ is well-defined. Clearly $ q $ is a seminorm. Also $ q $ is continuous as for all $ \lambda > 0 $, the preimage $ q^{-1} [0,\lambda) $ is a basic neighborhood of $ 0 $, e.g.\ 
	\[ q^{-1}[0,1) = \left( \bigcap_{i \in \Lambda } \mathcal{N}^{r_i} (0;A_i,(U_i,\phi_i),p_i,\epsilon_i) \right) \cap C_{\text{vS,c}}^\infty (M,E). \] So every basic neighborhood of 0 arises as a preimage of a continuous seminorm. Consequently, $ C_{\text{vS,c}}^\infty (M,E) $ is locally convex (see \cite[\S 18]{koethe}).
\end{proof}

We will now provide an alternative description of the topology on $ C_{\text{vS,c}}^\infty (M,E) $ as an inductive limit of certain locally convex spaces.
This characterization also implies that $ C_{\text{vS,c}}^\infty (M,E) $ is a locally convex space (thus providing an elegant proof of Proposition \ref{locally convex vs prop}).
Note however: Though the proof of Proposition \ref{locally convex vs prop} is a bit cumbersome, it is also completely elementary and does not use auxiliary results on inductive limits.

\begin{definition}
Let $K \subseteq M$ be a compact subset and $E$ be a locally convex space.
Then we define 
  \begin{displaymath}
  C^\infty_K (M,E) := \{f\in C^\infty (M,E) \mid \supp (f,0) \subseteq K\}                                                          
  \end{displaymath}
 and topologize this space with the compact open $C^\infty$-topology, i.e.\ the topology generated by the subbase $\mathcal{N} \cap C^\infty_K (M,E)$ where $\mathcal{N}$ runs through all elementary neighborhoods of $C^\infty_{\text{vS}} (M,E)$.
 Recall from \cite[Proposition 4.19]{glockomega} that $C^\infty_K (M,E)$ is a locally convex vector space.
\end{definition}

\begin{remark}
 Since all functions $C^\infty_K (M,E)$ have compact support contained in $K$ one can prove that the compact open $C^\infty$-topology coincides with the subspace topologies induced by $\vsmooth (M,E)$ and $\fsmooth (M,E)$.
 However, we will not need this. 
\end{remark}

Denote by $\mathcal{K} (M)$ the set of compact subsets of $M$. 
Observe that as sets $C^\infty_{\text{vS,c}} (M,E) = \bigcup_{K \in \mathcal{K} (M)} C^\infty_K (M,E)$.
We claim that the topology on the compactly supported functions is determined by the smaller locally convex spaces:
To see this, recall that with respect to inclusion, $\mathcal{K} (M)$ is a directed set. 
Further, for $K,L \in \mathcal{K} (M)$ with $K \subseteq L$ the canonical inclusion $\iota_K^L\colon C^\infty_K (M,E) \rightarrow C^\infty_L (M,E)$ is continuous linear by definition of the topology.
Hence we can form the locally convex inductive limit $\displaystyle \lim_{\rightarrow} C^\infty_K (M,E)$ (cf.\ \cite[\S 19 3.]{koethe}) of the family $\{C^\infty_K (M,E)\}_{\mathcal{K} (M)}$ (with respect to the canonical inclusions).

\begin{lemma}\label{lem: indlim}
 Let $E$ be a locally convex space, then as locally convex spaces
  \begin{displaymath}
   C^\infty_{\text{vS,c}} (M,E) = \lim_{\rightarrow} C^\infty_K (M,E).
  \end{displaymath}
\end{lemma}

\begin{proof}
 Since as sets $C^\infty_{\text{vS,c}} (M,E) = \displaystyle\lim_{\rightarrow} C^\infty_K (M,E)$, we only have to prove that the topologies coincide.
 However, since $M$ is $\sigma$-compact, \cite[Proposition 8.13 (d)]{glockomega} implies that a basis for the inductive limit topology on $C^\infty_{\text{vS,c}} (M,E) = \displaystyle\lim_{\rightarrow} C^\infty_K (M,E)$ is given by the basic neighborhoods of the very strong topology.
\end{proof}

\begin{proposition}
	For each class $ [f] $ in $ \fsmooth (M,E) $ define $ \phi_{[f]} \colon [f] \to C_{\text{vS,c}}^\infty (M,E) $ by $ \phi_{[f]}(g) = g-f $. 
	Then $ \mathcal{A} = \lbrace (\phi_{[f]},[f]) \rbrace_{f \in \smooth (M,E)} $ is a smooth atlas for $ \fsmooth (M,E) $. Hence $ \fsmooth(M,E) $ is a smooth manifold modeled on $ C_{\text{vS,c}}^\infty (M,E) $.
\end{proposition}
\begin{proof}
	We will first show that every chart $ \phi_{[f]} $ is a homeomorphism. First of all, note that $ \phi_{[f]} $ is well-defined since $ g - f $ is smooth and compactly supported for $ g \in [f] $. It is bijective with inverse $ \phi_{[f]}^{-1} (h) = h + f $. Both $ \phi_{[f]} $ and $ \phi_{[f]}^{-1} $ are continuous by Lemma \ref{continuous addition}.
	
	The chart domains of $ \mathcal{A} $ cover $ \fsmooth $, whence we have to check that chart transformations are smooth. Let $ (\phi_{[f]},[f]) $ and $(\phi_{[g]},[g]) $ be charts with $ [f] \cap [g] \neq \emptyset $.
	Then $ [f] = [g] $ and $ \phi_{[g]} \circ \phi_{[f]}^{-1}(h) = h + f - g $, whence it is smooth in $ h $ as addition in $C_{\text{vS,c}}^\infty (M,E)$ is so.
\end{proof}
Structurally, the manifold $ \fsmooth (M,E) $ is just a collection of (affine) copies of $ C_{\text{vS,c}}^\infty (M,E)$. For this reason, it is also called in \cite{michor} a \emph{local topological affine space}.

To construct a manifold structure on $ \fsmooth (M,X) $ for an arbitrary locally convex manifold $X$ one needs a so called \emph{local addition} on $X$ (cf.\ \cite{michor,KM97}).
A local addition replaces the vector space addition. 
It allows to ``smoothly choose'' charts on $X$ (see \cite{stacey} for more information). 
The details are similar to \cite[Section 10]{michor} but require certain analytical tools (e.g.\ a suitable version of the $\Omega$-Lemma, \cite[Appendix F]{glockomega})\footnote{To apply the $\Omega$-Lemma as stated in \cite{glockomega}, one needs a topology on spaces of compactly supported sections in vector bundles. In ibid.\ the compact open $C^\infty$-topology is used, however by arguments similar to Lemma \ref{lem: indlim} one proves that this topology coincides with the very strong topology.}.

\section{Application to bisection groups}
In this section we use our results on the very strong and the fine very strong topology to turn certain groups into topological groups.
The groups envisaged here are the bisection groups associated to certain Lie groupoids.
A reference on (finite-dimensional) Lie groupoids is \cite{mackenzie}, see \cite{Schmeding2015,SchmedingWockel15} for infinite-dimensional Lie groupoids.
\begin{definition}[Lie groupoid]
	Let $ M $ be a finite-dimensional smooth manifold and $ G $ a smooth manifold modeled on a locally convex vector space. Then a groupoid $ \mathcal{G} = (G \rightrightarrows M) $ with source projection $ \alpha \colon G \to M $ and target projection $ \beta \colon G \to M $ is a \emph{(locally convex) Lie groupoid} if $ \alpha $ and $ \beta $ are smooth submersions (i.e. locally projections), partial multiplication $ m \colon G \times_{\alpha,\beta} G \to G $ is smooth, object inclusion $ 1 \colon M \to G $ is smooth, and inversion $ i \colon G \to G $ is smooth.
\end{definition}
\begin{definition}[Bisection group]
	The \emph{group of bisections} $ \Bis (\mathcal{G}) $ of a Lie groupoid $ \mathcal{G} = (G \rightrightarrows M) $ is the set of sections $ \sigma \colon M \to G $ of $ \alpha $ such that $ \beta \circ \sigma $ is a diffeomorphism of $ M $. The group operation $ \star $ is given by $$ (\sigma \star \tau)(x) := \sigma ((\beta \circ \tau)(x)) \tau (x). $$ With this operation, the object inclusion $ 1 \colon M \to G $ becomes the neutral element and the inverse of a section $ \sigma $ is $ \sigma^{-1}(x) = i(\sigma((\beta \circ \sigma)^{-1}(x))). $
\end{definition}
\begin{example}
	\begin{enumerate}
		\item For a finite-dimensional manifold $ M $, the \emph{unit Lie groupoid} is the groupoid $ (M \rightrightarrows M) $ with both source and target projection $ \id_M $. The bisection group of this groupoid is trivial.
		\item Let $ M $ be a finite-dimensional smooth manifold. Then $ \mathcal{P}(M) := (M \times M \rightrightarrows M) $ with source projection $ \alpha = \pr_2 $ and target projection $ \beta = \pr_1 $ is a Lie groupoid. Multiplication in the groupoid is given by $ (x,y)(y,z) = (x,z) $. Postcomposition $ \beta_* $ induces an isomorphism $ \Bis (\mathcal{P}(M)) \cong \Diff(M) $ of groups, where \( \Diff (M) \) is the group of smooth diffeomorphisms of \( M \).
		\item Suppose $ G $ is a locally convex Lie group, and $ * $ is the one-point space. Then $ (G \rightrightarrows *) $ is a Lie groupoid with bisection group $ G $.
	\end{enumerate}
\end{example}

To prepare the construction of a topological group structure on bisection groups, recall the following facts on diffeomorphism groups of finite-dimensional manifolds.
\begin{remark}\label{rem: topgroup}
 Let $M$ be a finite-dimensional manifold and $\Diff (M)$ be the group of smooth diffeomorphisms of $M$.
 As $\Diff (M) \subseteq C^\infty (M,M)$, we can endow $\Diff (M)$ either with the subspace topology induced by the very strong topology (write $\Diff_{\text{vS}} (M)$) or with respect to the fine very strong topology (we write $\Diff_{\text{fS}} (M)$).
 Now as a consequence of \cite[Corollary 7.7]{michor} and Proposition \ref{prop: topologiescoincide}, both $\Diff_{\text{vS}} (M)$ and $\Diff_{\text{fS}} (M)$ are topological groups.
 Note that $\Diff_{\text{fS}} (M)$ is even a locally convex Lie group by \cite[Theorem 11.11]{michor}.
 In particular, we remark that the (subspace topology induced by the) fine very strong topology is the Lie group topology of $\Diff (M)$.
\end{remark}

\begin{proposition}\label{prop: topgp}
	If $ \Bis (\mathcal{G}) $ is equipped with the subspace topology with respect to $ \vsmooth (M,G) $ or $ \fsmooth (M,G) $, then $ \Bis (\mathcal{G}) $ becomes a topological group.
\end{proposition}
\begin{proof}
	We will prove that $ \Bis (\mathcal{G}) $ becomes a topological group when equipped with the subspace topology with respect to $ \vsmooth (M,G) $. The case where we consider the subspace topology with respect to $ \fsmooth (M,G) $ can be proven identically, since we only use results that hold in both topologies.
	Let $ \Omega \colon \Bis (\mathcal{G}) \times \Bis (\mathcal{G}) \to \Bis (\mathcal{G}) $ be the multiplication map defined by $ \Omega (\sigma,\tau) = \sigma \star \tau $, and let $ \iota $ be the inclusion $ \Bis (\mathcal{G})  \to \vsmooth (M,G) $. Observe that we can write $$ \Omega(\sigma,\tau)(x) = \sigma((\beta \circ \tau)(x))\tau(x) = m(\Gamma(\beta \circ \tau, \sigma)(x),\tau(x)). $$ So $ \iota \circ \Omega $ can be written as a composition of continuous maps; the diagram
	\begin{align*}
		\xymatrix{
			\Bis (\mathcal{G}) \times \Bis (\mathcal{G}) \ar[d]^{(\beta_* \circ \pr_2, \iota \times \iota )} \ar@{.>}[rr]^{\iota \circ \Omega} &&\vsmooth (M,G)	\\
			\Prop_{\textup{vS}} (M,M) \times \vsmooth(M,G) \times \vsmooth (M,G) \ar[d]^{\Gamma \times \id}  &&  \\
			\vsmooth (M,G) \times \vsmooth (M,G) \ar[rr]^-{\cong} & &  \vsmooth (M,G \times G) \ar[uu]^{m_*}
			}
	\end{align*}
	commutes. Here we have used that $ \beta_* (\Bis (\mathcal{G})) \subseteq \Diff (M) \subseteq \Prop (M,M) $ by definition of bisections. All of the maps represented by normal arrows in the diagram are continuous by results in the previous sections. Since $ \iota \circ \Omega $ is continuous, so is $ \Omega $.
	
	Let $ \Phi \colon \Bis (\mathcal{G}) \to \Bis (\mathcal{G}) $ be the inversion map. 
	Inversion $ \Inv \colon \Diff_{\text{vS}} (M) \to \Diff_{\text{vS}}(M) $ is continuous by \cite[Theorem 7.6]{michor} and Proposition \ref{prop: topologiescoincide}. The diagram
	\begin{align*}
		\xymatrix{
			\Bis(\mathcal{G}) \ar[d]^{\beta_*} \ar@{.>}[rr]^{\iota \circ \Phi} && \vsmooth(M,G) \\
			\Diff_{\text{vS}}(M) \ar[r]^{\Inv} & \Diff_{\text{vS}}(M) \ar[r]^{\sigma_*} & \vsmooth (M,G) \ar[u]^{i_*}
			}
	\end{align*}
	commutes as $ \Phi(\sigma)(x) = i(\sigma((\beta \circ \sigma)^{-1}(x))) = (i_*(\sigma_*(\Inv(\beta_*(\sigma))))(x)$. 
	All maps represented by normal arrows are continuous. Thus $ \iota \circ \Phi $ and also $ \Phi $ are continuous.
\end{proof}

As noted in Remark \ref{rem: topgroup}, $\Diff (M)$ is a topological group with respect to the subspace topologies induced by the (fine) very strong topology on $C^\infty (M,M)$.
Thus we obtain the following morphisms of topological groups.
\begin{corollary}
	The target projection \( \beta \colon G \to M \) of a locally convex Lie groupoid \( \mathcal{G} = (G \rightrightarrows M) \) induces a map \( \beta_* \colon \Bis (\mathcal{G}) \to \Diff(M) \) given by postcomposition. This is a homomorphism of topological groups with respect to the very strong and fine very strong topologies on both groups.
\end{corollary}
\begin{proof}
	Since $\beta_* \colon \vsmooth (M,G) \rightarrow \vsmooth (M,M)$ is continuous, so is the (co)restriction of \( \beta_* \) to $\Bis (\mathcal{G})$ and \( \Diff (M) \). The same argument holds in the fine very strong topology.
	The map \( \beta_* \) is also a group homomorphism, since 
	\[ \left( \beta_* ( \sigma \star \tau ) \right) (x) = \beta \left( \sigma ((\beta \circ \tau)(x)) \tau (x) \right) = \beta ( \sigma ( \beta \circ \tau ) (x)) = \left( \beta_*(\sigma) \circ \beta_* (\tau)\right) (x).\qedhere \]
\end{proof}

The results of this section enable the construction of a Lie group structure on $\Bis (\mathcal{G})$. 
It is worth noting that the key step in constructing the Lie group structure is sorting out the topology of the function spaces.
Using the manifold structure on $\fsmooth (M,G)$ (see comments in Section \ref{sect: smmfd}) one establishes the smoothness of joint composition and postcomposition with respect to these structures.
Since Theorem A and the $\Omega$-Lemma \cite[Appendix F]{glockomega} are at our disposal, one can copy exactly the arguments from the finite-dimensional case outlined in \cite[\S 10 and \S 11]{michor}.
After that one can proceed as in \cite{Schmeding2015} and establish smoothness of the group operations following the proof of Proposition \ref{prop: topgp}.
Again, results of this type are beyond the scope of the present paper. 

\newpage
\appendix
\section{Calculus in the locally convex setting}
\label{calculus}
In this appendix we lay down the definitions, notation and conventions used throughout this article regarding locally convex vector spaces, infinite-dimensional manifolds, and smooth maps between such objects.

\textbf{Locally convex vector spaces.}
\begin{definition}
	\begin{enumerate}
		\item A real vector space $ E $ is a \emph{topological vector space} if $ E $ is equipped with a Hausdorff topology turning both addition and scalar multiplication into continuous maps.
		\item A topological vector space $ E $ is called \emph{locally convex} if every 0-neighborhood contains a convex 0-neighborhood.
	\end{enumerate}
\end{definition}
Particularly useful for our purposes (constructing the very strong topology on a space of smooth functions into a locally convex vector space) is the perspective of a locally convex vector topology as generated by a suitable family of seminorms on the space.
\begin{definition}
\label{seminormfamilydef}
Let $ E $ be a locally convex vector space.
\begin{enumerate}
	\item A seminorm on $ E $ is a map $ p \colon E \to [0,\infty) $ such that $ p(x+y) \leq p(x) + p(y) $ and $ p(\lambda x) = |\lambda | p(x) $ for all $ x,y \in E $ and all $ \lambda \in \mathbb{R} $.
	\item A family $ \mathcal{P} $ of continuous seminorms on $ E $ is called \emph{generating} if $ E $ has the initial topology with respect to $ \mathcal{P} $.
	\item For a seminorm $ p $ on a vector space $ E $, $ x \in E $ and $ \epsilon > 0 $, we write $ B_\epsilon^p (x) := \lbrace y \in E : p(y-x)<\epsilon \rbrace $ and $ \overline{B}_\epsilon^p(x) := \lbrace y \in E : p(y-x) \leq \epsilon \rbrace $, called \emph{open} and \emph{closed} seminorm balls, respectively.
\end{enumerate}
\end{definition}
We list some properties of seminorms and families of such that are particularly useful for us.
\begin{proposition}[\S 18 in \cite{koethe}]
	\label{locally convex prop}
	Let $ E $ be a locally convex vector space.
	\begin{enumerate}
		\item\label{cont seminorms are generating} The family of all continuous seminorms on $ E $ is generating.
		\item\label{generating family criterion} A family $ \mathcal{P} $ of continuous seminorms on $ E $ is generating if and only if for any continuous seminorm $ p $ on $ E $ there exist $ p_1,\dots,p_n \in \mathcal{P} $ and a $ c \in \mathbb{R} $ such that $$ p \leq c \sup_{1\leq i \leq n} p_i. $$
		\item\label{locally convex product} Let $ \lbrace E_i \rbrace_{i\in I} $ be a family of locally convex vector spaces. The product $ \prod_{i \in I} E_i $ of vector spaces is a locally convex vector space when equipped with the product topology. Moreover, the set $ \bigcup_{i \in I} \lbrace p \circ \pr_i : \mbox{$p$ is a continuous seminorm on $E_i$} \rbrace $ is a generating family of seminorms for $ \prod_{i \in I} E_i $.
	\end{enumerate}
\end{proposition}
\textbf{Smooth maps between locally convex vector spaces.}
Differential calculus in the locally convex setting (in our sense) is sometimes known as \emph{Bastiani calculus} (see \cite{bastiani}) or \emph{Keller $C^r$-theory} (see \cite{keller}). 
The idea is to define the derivative of a map between locally convex vector spaces using directional derivatives.
\begin{definition}
	\label{Crmap}
	Let $ E $ and $ F $ be locally convex vector spaces, $ U \subseteq E $ open, and $ f \colon U \to F $ a map. 
	If it exists, we define for $ (x,h) \in U \times E $ \emph{the directional derivative}
	$$ \dd f(x,h) := D_h f(x) := \lim_{t\to 0} \frac{f(x+th)-f(x)}{t}. $$
	For $ r \in \mathbb{N}\cup \lbrace \infty \rbrace $ we say that $ f $ is $ C^r $ if
	$$ \dd^{(k)}f(x;y_1,\dots,y_k) := D_{y_k}D_{y_{k-1}}\cdots D{y_1} f(x) $$
	exists for all $ k $ with $ 1 \leq k \leq r $ and all $ y_i \in E $, the $ \dd^{(k)} f \colon U \times E^k \to F $ define continuous maps, and $ f $ is continuous. We say $f$ is \emph{smooth} if it is $C^\infty$.
	
	Elsewhere, $ \dd^{(k)} f (x;y_1,\dots,y_n) $ is often denoted
	\begin{align*}
		\frac{\partial^k}{\partial y_k \cdots \partial y_1}f(x) && \mbox{or} && \frac{\partial}{\partial y} f(x), \ \text{ where } y = (y_1,\dots,y_k) .
	\end{align*}
\end{definition}
Here is an alternative, but equivalent condition for when a map is $ C^r $ for $ r \geq 2 $.
\begin{lemma}[Lemma 1.14 in \cite{glock1}]
	Suppose $ f $ is $ C^1 $ in the sense of Definition \ref{Crmap}. Then $ f $ is a $ C^r $-map if and only if $ \dd f \colon U \times E \to F $ is $ C^{r-1} $. In this case, we define $ \dd^k f $ for $ k \leq r $ recursively by
	$$ \dd^k f := \dd^{k-1} (\dd f) \colon U \times E^{2^k-1} \to F .$$
\end{lemma}
By convention $ \dd^0 f = \dd^{(0)} f = f $.
Note that whereas $ \dd^{(k)} f $ is obtained by taking iterated directional derivatives of $ f $ with respect to the first variable, $ \dd^k f $ is the derivative of $ \dd^{k-1} f $ with respect to all variables. 
The $ \dd^k f $ and $ \dd^{(k)} f $ determine each other:
\begin{proposition}[Proof of Lemma 1.14 in \cite{glock1}]
	\label{longformula}
	Let $ k \in \mathbb{N} $. Then 
	\begin{enumerate}
		\item For each $ j \in \lbrace 1,\dots,k \rbrace $ and each integer $j$-tuple $\alpha $ with $ 1 \leq \alpha_1 < \cdots \alpha_j \leq 2^k - 1 $ there exists $ n_\alpha \in \mathbb{N}_0 $ such that for every $ C^k $-map $ f \colon E \supseteq U \to F $, we have
			$$ \dd^k f (x;y_1,\dots,y_{2^k-1}) = \sum_{j=1}^k \sum_{\alpha_1 < \cdots < \alpha_j} n_\alpha \dd^{(j)} f (x;y_{\alpha_1},\dots,y_{\alpha_j}) $$
		for all $ x \in U $ and all $ y_1,\dots,y_{2^k-1} \in E $.
	\item There exist numbers $ \alpha_1 < \cdots < \alpha_k $ in $ \lbrace 1, \dots , 2^k-1 \rbrace $ such that for every $ C^k $-map $ f\colon E \supseteq U \to F $, we have
		$$ \dd^{(k)} f (x;y_1,\dots,y_k) = \dd^k f (x;\iota_{\alpha} (y_1,\dots,y_k)) $$
		for all $ x \in U $ and all $ y_1,\dots,y_k \in E^k $. Here, $ \iota_\alpha \colon E^k \to E^{2^k-1} $ is the inclusion sending $ y_i $ to position $ \alpha_i $, i.e. $ \pr_j \iota_\alpha (y_1,\dots,y_k) = y_i $ if $ j = \alpha_i $ for some $ i $ and 0 otherwise.
	\end{enumerate}
\end{proposition}
\begin{proposition}[1.12, 1.13, and 1.15 in \cite{glock1}]
	\label{chainruleprop}
	Let $ E, F $ and $ G $ be locally convex vector spaces, $ U \subseteq E $ and $ V \subseteq F $ open, $ f \colon U \to F $ and $ g \colon V \to G $ maps such that $ f(U) \subseteq V $, and $ r \in \mathbb{N} \cup \lbrace \infty \rbrace $.
	\begin{enumerate}
		\item\label{prechainrule} If $ f $ and $ g $ are $ C^r $-maps, then $ g \circ f \colon U \to G $ is $ C^r $.
		\item\label{chainrule} If $ f $ and $ g $ are $C^1 $-maps, then $ \dd(g \circ f)(x,y) = \dd g(f(x),\dd f(x,y)) $ for all $ x \in U $ and all $ y \in E $.
		\item\label{schwarzrule} If $ f $ is $ C^r $ and $ x \in U $, then $ \dd^{(k)} f (x;\cdot) \colon E^k \to F $ is a symmetric $k$-linear map for every $ k \leq r $.
	\end{enumerate}
\end{proposition}
The chain rule \eqref{chainrule} is difficult to work with directly for higher order derivatives. But there is a way to circumvent this problem, by defining a map $ \T^j f $ in such a way that we have the identity $ \dd^j (g \circ f) = \dd^j g \circ \T^j f. $
\begin{definition}
	\label{Tdef}
	If $ f \colon E \supseteq U \to F $ is $ C^k $, we define
	\begin{align*}
		\T f \colon U \times E \to F \times F, \quad 		(x,y) \mapsto (f(x),\dd f(x,y))
	\end{align*}
	Note that $ \T f $ is $ C^{k-1} $. For $ j \leq k $ define $ \T^j f := \T (\T^{j-1}f) \colon U \times E^{2^j -1} \to F^{2^j}. $ We also define $ \tilde{\T}^j f $ as the projection of $ \T^j f $ onto the last $ 2^j - 1 $ coordinates, so that $$ \T^j f (x;y_1,\dots , y_{2^j - 1}) = (f(x),\tilde{\T}^j f (x;y_1,\dots,y_{2^j -1}) . $$
\end{definition}
\textbf{Manifolds modeled on locally convex vector spaces.}
\begin{definition}
	A $ C^0 $-manifold modeled on a locally convex vector space $ E $ is a Hausdorff topological space $ X $ which is locally homeomorphic to $ E $.
	
	By \emph{locally homeomorphic} to $ E $ it is meant that any point $ x \in X $ has an open neighborhood that is homeomorphic to an open set in $ E $.
\end{definition} 
\begin{remark}
	As opposed to the standard definition of finite-dimensional manifolds, we don't require that manifolds are $ \sigma $-compact. However, we always assume that finite-dimensional manifolds are $ \sigma $-compact. For locally compact and locally metrizable Hausdorff spaces (which includes finite-dimensional manifolds), being \( \sigma \)-compact is equivalent to being second countable.
\end{remark}
Just like in the finite-dimensional case, a $ C^r $ structure on an infinite-dimensional manifold $ X $ is given by a choice of a \emph{$C^r$-atlas}, or more precisely an equivalence class of such.
\begin{definition}
	Let $ X $ be a $ C^0 $-manifolds modeled on $ E $. A \emph{$ C^r $-atlas} for $ X $ is a collection $$ \mathcal{A} = \lbrace (\phi,U) : \mbox{$ U \subseteq X $ open, $ U' \subseteq E  $ open, and $ \phi \colon U \to U' $ a homeomorphism} \rbrace $$ such that for any two elements $ (\phi,U),(\psi,V) \in \mathcal{A} $ (called \emph{charts}) for which $ U \cap V \neq \emptyset $, the \emph{transition map} $ \psi \circ \phi_{|\phi(U\cap V)}^{-1} \colon \phi (U \cap V) \to V' $ is a $ C^r $-map.
\end{definition}
We define an equivalence class on the set of $ C^r $-atlases of $ X $ by declaring that $ \mathcal{A} \sim \mathcal{A}' $ if $ \mathcal{A} \cup \mathcal{A}' $ is a $ C^r $-atlas for $ X $. Using \eqref{prechainrule} and \eqref{chainrule} in Proposition \ref{chainruleprop} one may check that this is indeed an equivalence relation. Given a specified equivalence class $ [\mathcal{A}] $ of $ C^r $-atlases for $ X $, a chart in $ [\mathcal{A}] $ (or simply a chart for $ X $) is a chart in $ \mathcal{A}' $ for some $ \mathcal{A}' \in [\mathcal{A}] $.
\begin{definition}
	A $ C^r $-manifold $ X $ modeled on a locally convex vector space $ E $ (or just a locally convex $ C^r $ manifold) is a $ C^0 $-manifold modeled on $ E $ together with a chosen equivalence class of $ C^r $-atlases for $ X $.
\end{definition}

\begin{definition}[ \( C^r \)-map]
	Let $ f \colon X \to Y $ be a map. Then $ f $ is a \emph{$C^r$-map} if for all $ x \in X $ and all pairs of charts $ (U,\phi) $ and $ (V,\psi) $ on $ X $ and $ Y $, respectively, such that $ x \in U $ and $ f(x) \in V $, the composition $ \psi \circ f \circ \phi_{|\phi(f^{-1}(V))}^{-1} \colon \phi(f^{-1}(V)) \to \psi(V) $ is a $ C^r $-map.
	Denote the set of all $ C^r $-maps $ X \to Y $ by $ C^r (X,Y) $.
\end{definition}

Products of manifolds, tangent spaces etc.\ can be defined for manifolds modeled on locally convex spaces as in the finite-dimensional setting. 
More information on this and on locally convex Lie groups (see below) can be found in \cite{neeb}.
\begin{definition}[Locally convex Lie group]
	Let $ G $ be a smooth manifold modeled on a locally convex vector space which is also a group. If these structures are compatible in the sense that group multiplication $ G \times G \to G $ and inversion $ G \to G $ are smooth, then $ G $ is a \emph{locally convex Lie group}. 
\end{definition}

\section{Auxiliary results for Sections \ref{sect: vsTop} and \ref{sect: compo} }
\label{app: details}

In this appendix we provide details and auxiliary results for the first two sections.
\begin{lemma}\label{greenlemma}
	Let $ m,n \in \mathbb{N}$, $ r \in \mathbb{N}_0$, $ A \subseteq \mathbb{R}^m $ compact, \( E \) a locally convex vector space, $ p $ a continuous seminorm on $ E $, and $ R > 0 $.
	
	Then there exists $ K > 0 $ such that for all smooth maps $f \colon \mathbb{R}^m \rightarrow \mathbb{R}^n$ and $ g\colon \mathbb{R}^n \rightarrow E$ for which $ \Vert f \Vert (r,A) \leq R $, one has
	$
	\Vert g \circ f \Vert (r,A,p) \leq K \Vert g \Vert (r, f(A), p) \Vert f \Vert (r,A).
	$
\end{lemma}
\begin{proof}
The proof is naturally divided into two parts.

\textbf{Part 1.}
Our first goal is to control the size of $ \T^k f $ (see Definition \ref{Tdef}). Let $ m, n $ and $ r $ be given. The precise statement we want to show is that for all $ k $ with $ 0 \leq k \leq r $ and all $ a \in A $, and for all $ z_1, z_2, \dots , z_{2^k - 1} \in \lbrace 0, e_1, \dots , e_m \rbrace $ one has 
\begin{displaymath}
\Vert \T^k f (a; z_1, z_2, \dots , z_{2^k - 1}) \Vert_{\infty} \leq c \Vert f \Vert (r, A) =: C,
\end{displaymath}
where $ c \geq 1 $ is a constant (depending only on $ r $, $ m $ and $ n $ ) such that 
\begin{displaymath}
\Vert \dd^k f (a;z_1,\dots,z_{2^k -1}) \Vert_\infty \leq c \Vert f \Vert (r, A).
\end{displaymath}
Such a constant $ c $ exists by Proposition \ref{longformula}. We prove the seemingly more general statement that for all non-negative integers $ i $ and $ j $ such that $ 0 \leq i + j \leq r $ one has
\begin{equation}
\label{inductioninequality}
\Vert \dd^i \T^j f(a;z_1,\dots,z_{2^{i+j}-1}) \Vert_\infty \leq C,
\end{equation}
by induction on $ i + j $.
When $ i + j = 0 $, both $ i = 0 $ and $ j = 0 $, so the left hand side in inequality \eqref{inductioninequality} is $ \Vert f(a) \Vert_\infty \leq C $, by the choice of $ c $.

Now let $ k $ be a positive integer with $ 0 \leq k \leq r-1 $ and assume that inequality \eqref{inductioninequality} holds for all non-negative integers $ i $ and $ j $ with $ i+j \leq k $. If $ i + j = k + 1 $ and $ 1 \leq j $, then 
\begin{align*}
& \Vert \dd^i \T^j f(a;z_1,\dots,z_{2^{i+j}-1}) \Vert_\infty \\
	=& \Vert \left( \dd^i (\T^{j-1} f (a;z_1,\dots,z_{2^{i+j-1}-1}), \dd \T^{j-1} f (a;z_1,\dots,z_{2^{i+j}-1}) \right) \Vert_\infty \\
=& \left\Vert \left( \dd^i \T^{j-1} f(a;z_1,\dots,z_{2^{i+j-1}-1}) , \dd^{i+1} \T^{j-1} f(a;z_1,\dots,z_{2^{i+j}-1}) \right) \right\Vert_\infty \\
\leq& \max \left( \Vert \dd^i \T^{j-1} f(a;z_1,\dots,z_{2^{i+j-1}-1}) \Vert_\infty , \Vert \dd^{i+1} \T^{j-1} f(a;z_1,\dots,z_{2^{i+j}-1}) \Vert_\infty \right),
\end{align*}
by definition of $ \T $ and the formula $ \dd (h_1,h_2) = (\dd h_1 , \dd h_2) $. By the induction hypothesis, $ \Vert \dd^i \T^{j-1} f \Vert_\infty \leq C $, so it remains to check that $ \Vert \dd^{i+1} \T^{j-1} f \Vert_\infty \leq C $. This last inequality is just what we started with, but with $ i $ increased by one and $ j $ decreased by one. Iterating the argument above, we are left with the case $ i = k + 1 \leq r $ and $ j = 0 $. This case holds by the choice of $ c $.	

\textbf{Part 2.}
Given \( 0 \leq k \leq r \), \( a \in A \), and $ y_1, \dots , y_k \in \lbrace e_1, \dots , e_m \rbrace \subset \mathbb{R}^m $, we would like to control the size of $\dd^{(k)} (g \circ f) (a;y_1, \dots , y_k).$
By Proposition \ref{longformula} there exist $ z_1, \dots , z_{2^k - 1} \in \lbrace 0 , y_1, \dots , y_k \rbrace $ such that
\begin{align*}
&  \dd^{(k)} (g \circ f) (a;y_1,\dots,y_k) = \dd^k (g \circ f) (a;z_1, \dots , z_{2^k - 1}) \\
=& \dd^k g \circ \T^k f (a;z_1,\dots,z_{2^k - 1}) = \dd^k g (f(a); \tilde{\T}^k f (a;z_1, \dots,z_{2^k-1}) \\
=& \sum_{j=1}^k \sum_{1\leq \alpha_1 < \cdots < \alpha_j \leq 2^k - 1} n_\alpha \dd^{(j)} g(f(a); \pr_\alpha (\tilde{\T}^k f (a;z_1, \dots,z_{2^k-1})) ),
\end{align*}
where the second sum is taken over all increasing $ j $-tuples $ \alpha = (\alpha_1,\dots,\alpha_j) $, each of the $ n_\alpha $ are fixed numbers as in Proposition \ref{longformula}, $ \pr_\alpha (t_1,\dots,t_{2^k-1}) = (t_{\alpha_1},\dots,t_{\alpha_j}) $, and $ \tilde{\T} $ is as in Definition \ref{Tdef}.
There exist $ \lambda_{j,i} \in \mathbb{R} $ such that
\begin{displaymath}
\tilde{\T}^k f (a;z_1,\dots,z_{2^k - 1} ) = \left( \sum_{i=1}^n \lambda_{1,i} e_i , \dots , \sum_{i=1}^n \lambda_{2^k-1,i}e_i \right).
\end{displaymath}
Note that every $ |\lambda_{i,j}| \leq C $ by part 1 of the proof and our choice of norm $ \Vert\cdot \Vert_\infty $! We have
\begin{align*}
& \dd^{(j)} g(f(a);\pr_\alpha \tilde{\T}^k f(a;z_1,\dots,z_{2^k-1}) \\
=& \sum_{(i_1,\dots,i_j) \in [n]^j} \lambda_{\alpha_1,i_1} \lambda_{\alpha_2, i_2} \cdots \lambda_{\alpha_j,i_j} \dd^{(j)} g (f(a);e_{i_1},\dots,e_{i_j})
\end{align*}
by \eqref{schwarzrule} in Proposition \ref{chainruleprop}.
Recall that $ p $ is a seminorm on $ E $. From the preceding paragraphs we have
\begin{align*}
& p \left( \dd^{(k)} (g \circ f) (a;y_1,\dots,y_k) \right) \\
=& p \left( \sum_{j=1}^k \sum_{1\leq \alpha_1 < \cdots < \alpha_j \leq 2^k - 1} n_\alpha \sum_{(i_1,\dots,i_j) \in [n]^j} \lambda_{\alpha_1,i_1} \lambda_{\alpha_2, i_2} \cdots \lambda_{\alpha_j,i_j} \dd^{(j)} g (f(a);e_{i_1},\dots,e_{i_j}) \right) \\
\leq& \sum_{j,\alpha,i} |n_\alpha||\lambda_{\alpha_1,i_1}|\cdots |\lambda_{\alpha_j,i_j}| p \left( \dd^{(j)} g(f(a);e_{i_1},\dots,e_{i_j}) \right) \\
\leq& \left( \sum_{j,\alpha,i} |n_\alpha| C^j \right) \Vert g \Vert (r,f(A),p) \leq \left( \sum_{j,\alpha,i} |n_\alpha| c^j R^{j-1} \right) \Vert f \Vert (r,A) \Vert g \Vert (r,f(A),p),
\end{align*}
where the last inequality holds since $ C = c\Vert f \Vert (r,A) \leq cR $ (recall that $ \Vert f \Vert (r,A) \leq R $ by assumption). Setting $ K = \sum |n_\alpha|c^j R^{j-1} $, we obtain
\begin{displaymath}
p \left( \dd^{(k)} (g \circ f) (a;y_1,\dots,y_k) \right) \leq K \Vert f \Vert (r,A) \Vert g \Vert (r,f(A),p).
\end{displaymath}
By definition of $ \Vert g \circ f \Vert (r,A,p) $ the assertion follows.
\end{proof}

We now prove that for spaces of smooth functions to locally convex spaces it suffices to consider basic neighborhoods with respect to the identity charts.

\begin{lemma}\label{lem: atlaschoice}
  Let $M$ be a finite-dimensional manifold and $E$ be a locally convex space. 
  
  Consider the topology $\mathcal{T}$  on $C^\infty (M,E)$ generated by basic neighborhoods which arise from elementary neighborhoods of the form $\mathcal{N}^{r} (f; A,(U, \phi),(E,\id_E),q, \epsilon)$.
  Then $\mathcal{T}$ coincides with the very strong topology on $C^\infty (M,E)$.	
\end{lemma}

\begin{proof}
 Restricting our choice of charts on $E$ to the canonical chart clearly generates a topology which is coarser than the very strong topology. 
 
 To see that it is also finer, consider first an arbitrary elementary neighborhood $\mathcal{N} = \mathcal{N}^{r} (f; A,(U, \phi),(V,\psi),q, \epsilon)$ in $\vsmooth (M,E)$.
 Recall that the compact open $C^\infty$-topology (see \cite[Defintion I.5.1]{neeb}) allows us to control the derivatives of functions on compact sets.
 We denote by $C^\infty (M,E)_{\text{co}}$ the vector space with this topology and note that the elementary neighborhoods of the very strong topology form a subbase of the compact open $C^\infty$-topology.
 
 Now choose a compact neighborhood $C \subseteq U$ of $A$ such that $f(C) \subseteq V$. 
 We endow the subset $\lfloor C,\psi(V)\rfloor := \{g\in C^\infty (M,E) \mid g(C) \subseteq \psi(V)\}$ with the subspace topology induced by the compact open $C^\infty$-topology.
 As $\psi (V) \subseteq E$ is open, we note that $\lfloor C,\psi(V)\rfloor$ is open in $C^\infty (M,E)_{\text{co}}$.
 Now \cite[Proposition 4.23 (a)]{glockomega} shows that 
 \begin{displaymath}
  (\psi^{-1})_* \colon \lfloor C,\psi(V) \rfloor \rightarrow C^\infty (\interior C,E)_{\text{co}},\quad  h \mapsto \psi^{-1} \circ h
 \end{displaymath}
 is continuous. 
 Moreover, the set $\mathcal{N}_C := \mathcal{N}^{r} (\psi \circ f|_{\interior C}; A,(U\cap \interior C, \phi),(E,\id_E),q, \epsilon)$ is open in $C^\infty (\interior C,E)_{\text{co}}$.
 Using the description of subbasic neighborhoods of the compact open $C^\infty$-topology in \cite{glock1}, we obtain elementary neighborhoods 
 \begin{equation}\label{eq: verify}
  \bigcap_{k=1}^N \mathcal{N}^{r} (f; A_k,(U_k, \phi_k),(E,\id_E),q_k, \epsilon) \subseteq  ((\psi^{-1})_*)^{-1} (\mathcal{N}_C) = \mathcal{N} \cap \lfloor C,V\rfloor \subseteq \mathcal{N}.
 \end{equation}
 Observe that the compact sets $A_k$ are contained by construction in $\interior C$.\footnote{We invite the reader to verify this in the proof of \cite[Proposition 4.23 (a)]{glockomega}.}
 Summing up, \eqref{eq: verify} together with Lemma \ref{trianglelemma} shows that every elementary neighborhood in $C^\infty (M,E)$ is open in $\mathcal{T}$.
 
 It is now easy to prove that every basic neighborhood is open in $\mathcal{T}$. 
 To this end let $\mathcal{M} = \bigcap_{k \in \mathbb{N}} \mathcal{N}_k$ be a basic neighborhood around $f \in C^\infty (M,E)$ with $\{A_k\}_{k \in \mathbb{N}}$ its the underlying compact family.
 Use \cite[30.C.10]{cech} to construct for every $A_k$ a compact neighborhood $C_k$ such that $\{C_k\}_{k \in \mathbb{N}}$ is locally finite.
 Now we proceed for every elementary neighborhood $\mathcal{N}_k$ as above (working with $C_k$!). 
 Since the family $\{C_k\}_{k \in \mathbb{N}}$ is locally finite, we thus end up with a basic neighborhood around $f$ which is contained in $\mathcal{T}$.
 As all basic neighborhoods of the very strong topology are open in $\mathcal{T}$, the topology $\mathcal{T}$ is finer as the very strong topology.
 This proves the assertion.
\end{proof}

\begin{lemma}\label{lem: atlaschoice2}
  Let $n \in \mathbb{N}$ and $E$ be a locally convex space. 
  Consider the topology $\mathcal{T}$ on $C^\infty (\mathbb{R}^n,E)$ generated by basic neighborhoods which arise from elementary neighborhoods of the form $\mathcal{N}^{r} (f; A,(\mathbb{R}^n, \id_{\mathbb{R}^n}),(E,\id_E),q, \epsilon)$.
  Then $\mathcal{T}$ coincides with the very strong topology on $C^\infty (M,E)$.	
\end{lemma}

\begin{proof}
 Clearly restricting our choice of charts to $\id_{\mathbb{R}^n}$ generates a topology coarser than the very strong topology. 
 To see that $\mathcal{T}$ is also finer, consider an elementary neighborhood $\mathcal{N} = \mathcal{N}^{r} (f; A,(U, \phi),(E,\id_E),q, \epsilon)$ in $\vsmooth (\mathbb{R}^n,E)$.
 
 Now $g \in \mathcal{N}$ if and only if $\Vert f\circ \phi^{-1}-g\circ \phi^{-1}\Vert (r,\phi(A),q) < \epsilon$.
 Using Lemma \ref{greenlemma}, we obtain a constant $K>0$ with 
 \begin{align*}
  \Vert f\circ \phi^{-1}-g\circ \phi^{-1}\Vert (r,\phi(A),q) &= \Vert (f-g)\circ \phi^{-1}\Vert (r,\phi(A),q)\\ &\leq K \Vert f-g\Vert (r,A,q) \Vert \phi^{-1}\Vert (r,\phi(A))
 \end{align*}
  Together with Lemma \ref{trianglelemma} this implies that for every $h \in \mathcal{N}$ there is $\delta >0$ with $\mathcal{N}_h := \mathcal{N}^{r} (h; A,(\mathbb{R}^n, \id_{\mathbb{R}^n}),(E,\id_E),q, \delta) \subseteq \mathcal{N}$.
  Hence $\mathcal{N}$ is also open in $\mathcal{T}$. As $\mathcal{N}_h$ controls functions on the same compact set as $\mathcal{N}$, the argument generalizes to basic neighborhoods. 
  Summing up, basic neighborhoods are open in $\mathcal{T}$ which is thus finer as the very strong topology. 
\end{proof}
Finally, we provide a standard fact about \( \sigma \)-compact topological spaces.
\begin{lemma}
	\label{dugundjifact}
	Let $ M $ be a \( \sigma \)-compact topological space. Then \( M \) admits a locally finite exhaustion by compact sets, i.e. there exists a locally finite family \( \lbrace A_i \rbrace_{i \in \mathbb{N}} \) of compact subsets of \( M \) such that \( \bigcup_{i \in \mathbb{N}} A_i = M \).
\end{lemma}
\begin{proof}
	By \cite{dugundji}[Theorem XI.7.2] there exists a family $ \lbrace K_i \rbrace_{i \in \mathbb{N}} $ of compact subspaces of $ M $ such that $ K_i \subseteq \interior K_{i+1} $ for all $ i \in \mathbb{N} $ and $ \bigcup K_i = M $. Now define $ A_1' = K_1 $ and $ A_i' = K_i \setminus \interior K_{i-1} $ for $ i \geq 2 $. 
	The resulting family $ \lbrace A_i' \rbrace_{i \in \mathbb{N} } $ is a locally finite exhaustion of \( M \) by compact sets.
\end{proof}

\section{Alternative description of the topology via jet bundles}
\label{folklore}
In this section we give a proof of the ``folklore'' fact that Michor's \( \mathcal{D} \)-topology on \( \smooth (M,N) \) (cf. \cite[Section 1]{michor}) coincides with the very strong topology when the target \( N \) is finite-dimensional. This is claimed in many places in the literature, e.g.\ \cite{hirsch} (where the $\mathcal{D}$-topology is called strong Whitney topology), but the authors of this paper have been unable to locate a proof. The \( \mathcal{D} \)-topology is constructed using jet bundles (cf.\ \cite[Section 1]{michor} and \cite[Section 41]{KM97}). We will briefly recall this construction here.

\begin{definition}
	Let \( U \subseteq \RR^m \) and \( V \subseteq \RR^n \) be open. For \( 0 \leq r \leq \infty \) we define the space of \emph{\( r \)-jets} from \( U \) to \( V \) by
	\[
		J^r (U,V) := U \times V \times  \prod_{i=1}^r L^{i}_{\text{sym}}(\RR^m, \RR^n),
	\]
	where \( L^{i}_{\text{sym}} (\RR^m,\RR^n) \) is the space of symmetric linear maps \( (\RR^m)^i \to \RR^n \). We topologize \( \prod_{i=1}^r L^{i}_{\text{sym}}(\RR^m , \RR^n ) \) with respect to the operator norm for multilinear maps. 

	For a smooth map \( f \colon U \to V \) we define
	\[
		j^rf(x) := \left( x, f(x), \dd^{(1)} f (x;\cdot), \frac{1}{2!} \dd^{(2)} f(x;\cdot), \dots, \frac{1}{r!} \dd^{(r)} f(x;\cdot) \right), 
	\]
	called the \emph{ \( r \)-jet} of \( f \) at \( x \).
\end{definition}
\begin{definition}
	Let \( M \) and \( N \) be finite-dimensional smooth manifolds. We define an \( r \)-jet from \( M \) to \( N \) to be an equivalence class of pairs \( (f,x) \), where \( f \colon M \to N \) is a smooth map and \( x \in M \). Two pairs \( (f,x) \) and \( (f',x') \) are equivalent if \( x = x' \) and \( T_x^r f = T_x^r f' \), where \( T^r \) is the \( r \)-th tangent mapping. We write \( j^r f(x) \) for the equivalence class of \( (f,x) \).
	The set of all \( r \)-jets from \( M \) to \( N \) is denoted \( J^r (M,N) \).
\end{definition}
In the case \( M = \RR^m \) and \( N = \RR^n \), the \( r \)-jet of a smooth map \( f \colon M \to N \) is represented by the Taylor polynomial of \( f \). So the different definitions above coincide.

If \( f \colon M \to M' \) is smooth, then there is an induced map \( J^r (N,f) \colon J^r (N,M) \to J^r (N,M') \) given by \( J^r (N,f) (j^r g(x)) = j^r (f \circ g)(x) \). 
For a diffeomorphism \( f \) we obtain in addition the map \( J^r (f,N) \colon J^r (M',N) \to J^r (M,N) \) given by \( J^r (f,N) (j^r g(x)) = j^r (g \circ f)(f^{-1}(x)) \).

The set of all \( r \)-jets \( J^r (M,N) \) form a smooth manifold. Suppose \( (U,\phi) \) is a chart for \( M \) and \( (V,\psi) \) is a chart for \( N \).
The map
\begin{align*}
	J^r (\phi^{-1},\psi) \colon J^r (U,V) \to& J^r \left( \phi(U),\psi(V) \right), \\
	J^r (\phi^{-1},\psi) &:= J^r (\phi^{-1} ,\psi(V)) \circ J^r (U,\psi) = J^r (\phi(U),\psi) \circ J^r (\phi^{-1} ,V) 
\end{align*}
is bijective. Note that \( J^r (U,V) \) can be identified with a subset of \( J^r (M,N) \), namely
\[
	J^r (U,V) = \left\lbrace f \in J^r (M,N) : f(U) \subseteq V \right\rbrace
\]
since at each point there exists a smooth map with a given Taylor expansion. The collection \( \left\lbrace \left( J^r (U,V) , J^r (\phi^{-1} , \psi ) \right) \right\rbrace \), where \( (U,\phi) \) and \( (V,\psi) \) runs through the charts for \( M \) and \( N \), respectively, is a smooth atlas for \( J^r (M,N) \). See \cite[1.8]{michor} for details.

Via jet bundles one can define a topology on the space of smooth functions.
Using the embedding $j^\infty \colon C^\infty (M,N) \rightarrow C(M,J^\infty (M,N))$ for $M,N$ finite-dimensional, one pulls back a certain topology to obtain Michor's $\mathcal{D}$-topology on $C^\infty (M,N)$.
We will not describe this construction in detail and refer to \cite[ 4.7.2]{michor} for the following alternative description of the topology

\begin{definition}\label{defn: Dtop}
 Let $M$ and $N$ be finite-dimensional smooth manifolds. 
 We define  
 \begin{displaymath}
	 \Omega (\mathbf{L}, \mathbf{U}) := \{f \in C^\infty (M,N) \mid j^{n} f(L_n) \subseteq U_n \mbox{ for all \( n \in \mathbb{N} \)}\}
 \end{displaymath}
	where $\mathbf{L} = \lbrace L_n\rbrace_{n \in \mathbb{N}}$ is a locally finite family of closed subsets of $M$, and $\mathbf{U} = \lbrace U_n \rbrace_{n\in \mathbb{N}}$ is a family of open subsets $U_n \subseteq J^{n} (M,N)$.
 
	Then the family of sets $\Omega (\mathbf{L}, \mathbf{U})$ where $\mathbf{L}$ runs through all locally finite families of closed sets and $\mathbf{U}$ runs through all families of open subsets $U_n \subseteq J^{n} (M,N)$
 is a basis for a topology on $C^\infty (M,N)$.
 Following Michor, we call this topology the \emph{$\mathcal{D}$-topology} and denote by $C^\infty_{\mathcal{D}} (M,N)$ the smooth functions with this topology.
\end{definition}
\begin{remark}
	Let \( \mathbf{L} = \lbrace L_n \rbrace_{n \in \mathbb{N}} \) be a locally finite family of closed subsets of \( M \), \( \mathbf{r} = \lbrace r_n \rbrace_{n\in\mathbb{N}} \) be a sequence of natural numbers, \( \mathbf{U} = \lbrace U_n \rbrace_{n \in \mathbb{N}} \) be a family of open subsets \( U_n \subseteq J^{r_n} (M,N) \), and set
	\[
		\Omega (\mathbf{r}, \mathbf{L}, \mathbf{U}) := \lbrace f \in \smooth (M,N) : j^{r_n} f(L_n) \subseteq U_n \mbox{ for all \( n \in \mathbb{N} \)} \rbrace.
	\]
	Taking the family of all sets \( \Omega (\mathbf{r}, \mathbf{L}, \mathbf{U} ) \) as \( \mathbf{L} \), \( \mathbf{r} \), and \( \mathbf{U} \) vary also forms a basis for the \( \mathcal{D} \)-topology on \( \smooth (M,N) \). To see this, note that if \( U_n \subseteq J^n (M,N) \) is open, \( k \geq 0 \), and \( \pi \colon J^{n+k} (M,N) \to J^n (M,N) \) is the ``truncation map'', then for all \( f \colon M \to N \) smooth and \( x \in M \), we have \( j^n f(x) \in U_n \) if and only if \( j^{n+k} f(x) \in \pi^{-1} U_n \).
\end{remark}

\begin{lemma}\label{lem: hacking}
	The family of all sets \( \Omega ( \mathbf{r},\mathbf{L}, \mathbf{U} ) \) is a basis for the $\mathcal{D}$-topology \( \smooth_{\mathcal{D}} (M,N) \) if we require that:
	\( \mathbf{r} = \lbrace r_n \rbrace_{n \in \mathbb{N}} \) runs through all sequences of natural numbers, \( \mathbf{L} = \lbrace L_n \rbrace_{n \in \mathbb{N}} \) runs through all locally finite families of \emph{compact} subsets of \( M \) such that each \( L_n \) is contained in some chart $(V_n,\psi_n)$ for \( M \), and \( \mathbf{U} = \lbrace U_n \rbrace_{n \in \mathbb{N}} \) runs through all families of open subsets \( U_n \subseteq J^{r_n} (V_n,W_n) \) where $(W_n, \phi_n)$ is a chart of $N$, 
\end{lemma}
\begin{proof}
	It suffices to show the following: Given a locally finite family \( \mathbf{L} = \lbrace L_n \rbrace_{n \in \mathbb{N}} \) of closed subsets of \( M \) and a family \( \mathbf{U} = \lbrace U_n \rbrace_{n\in \mathbb{N}} \) of open subsets \( U_n \subseteq J^n (M,N) \), there exist a sequence \( \mathbf{r} = \lbrace r_n \rbrace_{n \in \mathbb{N}} \) of natural numbers, a locally finite family \( \mathbf{A} = \lbrace A_n \rbrace_{n\in \mathbb{N}} \) of compact subsets of \( M \), a family \( \lbrace (V_n , \psi_n ) \rbrace_{n \in \mathbb{N}} \) of charts for \( M \) such that \( A_n \subseteq V_n \) for all \( n \in \mathbb{N} \), a family \( \lbrace (W_n , \phi_n ) \rbrace_{n \in \mathbb{N}} \) of charts for \( N \), and a family \( \mathbf{U}' = \lbrace U_n' \rbrace_{n \in \mathbb{N}} \) of open subsets \( U_n' \subseteq J^{r_n} (V_n, W_n) \) such that \( \Omega (\mathbf{r},\mathbf{A},\mathbf{U}') = \Omega (\mathbf{L}, \mathbf{U}) \).

	Since \( M \) is \( \sigma \)-compact there exists a compact exhaustion \( \mathbf{K} = \lbrace K_n \rbrace_{n \in \mathbb{N}} \) of \( M \), and by dividing each \( K_n \) into finitely many compact subsets each contained in some manifold chart, we may assume that each \( K_n \subseteq V_n \) for some chart \( (V_n, \psi_n ) \) for \( M \). We may further require that there is a chart $(W_n,\phi_n)$ such that $f(K_n) \subseteq W_n$. The family \( \lbrace K_i \cap L_k \rbrace_{(i,k) \in \mathbb{N} \times \mathbb{N}} \) is a locally finite family of compact sets since the \( K_i \) are compact and \( \mathbf{L} \) is a locally finite family of closed sets. Note that \( K_i \cap L_k \subseteq V_i \) for all \( (i,k) \in \mathbb{N} \times \mathbb{N} \). Moreover, for all \( k \in \mathbb{N} \) we have \( j^k f ( K_i \cap L_k ) \subseteq U_k \cap J^{r_k} (V_i,W_i)\) for all \( i \in \mathbb{N} \) if and only if \( j^k f (L_k) \subseteq U_k \).

	Take a bijection \( b \colon \mathbb{N} \times \mathbb{N} \to \mathbb{N} \), set \( A_{b(i,k)} = K_i \cap L_k \), and \( r_{b(i,k)} = k \), and \( U_{b(i,k)}' = U_k \cap J^{r_k} (V_i,W_i) \). Then
	\begin{align*}
		& \Omega (\mathbf{r},\mathbf{A}, \mathbf{U}' ) \\
		=& \lbrace f \in \smooth (M,N) : j^{r_n} f (A_n) \subseteq U_n' \mbox{ for all \( n \in \mathbb{N} \)}\rbrace \\
		=& \lbrace f \in \smooth (M,N) : j^{k} f (K_i \cap L_k) \subseteq U_k \cap J^{r_k} (V_i,W_i) \mbox{ for all \( i \in \mathbb{N} \) and all \( k \in \mathbb{N} \)} \rbrace \\
		=& \lbrace f \in \smooth (M,N) : j^k f(L_k) \subseteq U_k \mbox{ for all \( k \in \mathbb{N} \)} \rbrace \\
		=& \Omega (\mathbf{L},\mathbf{U}). \qedhere
	\end{align*}
\end{proof}
\begin{definition}\label{defn: elD}
     Let $ M $ and $ N $ be finite-dimensional smooth manifolds.
     If \( r \in \mathbb{N} \), \( A \subseteq M \) a compact subset, and \( O \subseteq J^r (M,N) \) open, then
	\[
		\mathcal{M}^r (A,O) := \lbrace f \in \smooth (M,N) : j^r f (A) \subseteq O \rbrace
	\]
	is an \emph{elementary \( \mathcal{D} \)-neighborhood}. 
	
	These neighborhoods play the same role for the \( \mathcal{D} \)-topology as the elementary (vS)-neighborhoods play in the very strong topology. If \( \mathbf{r} \) is a sequence of natural numbers, \( \mathbf{A} \) is a locally finite family of compact subsets of \(M \), and \( \mathbf{O} \) is a family of open subsets \( O_n \subseteq J^{r_n} (M,N) \), then 
	\[
		\Omega ( \mathbf{r}, \mathbf{A}, \mathbf{O} ) = \bigcap_{n \in \mathbb{N}} \mathcal{M}^{r_n} (A_n,O_n).
	\]
\end{definition}

\begin{lemma}
	\label{biglemmaD}
	Let \( M \) and \( N \) be finite-dimensional smooth manifolds. Suppose \( U \subseteq M \) and \( V \subseteq N \) are open. Consider the subspace \( C_{\mathcal{D},\text{sub}}^\infty (U,V) = \lbrace f \in \smooth (M,N) : f(U) \subseteq V \rbrace \subseteq C_{\mathcal{D}}^\infty (M,N) \).
	\begin{enumerate}
		\item \(  C_{\mathcal{D},\text{sub}}^\infty (U,V) \) is an open subset of \( C_{\mathcal{D}} (M,N) \).
		\item The restriction \( \res_{\mathcal{D}} \colon  C_{\mathcal{D},\text{sub}}^\infty (U,V) \to C_{\mathcal{D}}^\infty (U,V) \) is continuous.
		\item For an elementary $\mathcal{D}$-neighborhood \( \mathcal{M}^r (A,O) = \lbrace f \in \smooth (U,V) : j^r f (A) \subseteq O \rbrace \) in $ C_{\mathcal{D}}^\infty (U,V)$, we have 
		$\res_{\mathcal{D}}^{-1} \left( \mathcal{M}^r (A,O) \right) = \mathcal{M}^r (A,O)\subseteq C^\infty_{\mathcal{D}}(M,N)$.
	\end{enumerate}
\end{lemma}
\begin{proof}
	\begin{enumerate}
		\item By Lemma \ref{dugundjifact} there exists a locally finite exhaustion \( \lbrace A_n \rbrace_{n \in \mathbb{N}} \) of \( U \) by compact sets. Observe that
			\[
				C_{\mathcal{D},\text{sub}}^\infty (U,V) = \left\lbrace f \in \smooth(M,N) : j^0 f(A_n) \subseteq U \times V \mbox{ for all \( n \in \mathbb{N} \)} \right\rbrace.
			\]
		\item The inclusions $\iota_U \colon U \rightarrow M$ and $\iota_V \colon V \rightarrow N$ are continuous embeddings. 
			Consider the subspace \( \hat{C}_{\mathcal{D},\text{sub}}^\infty (U,V) := \lbrace f \in C_{\mathcal{D}}^\infty (U,N) : f(U) \subseteq V \rbrace \), and the corestriction \( R \colon \hat{C}_{\mathcal{D},\text{sub}}^\infty (U,V) \to C_{\mathcal{D}}^\infty (U,V), \quad f \mapsto f|^V \).\\
			Then \( \res_{\mathcal{D}} = R \circ (\iota_U)^* \). 
			The map $(\iota_U)^*$ is continuous with respect to the $\mathcal{D}$-topology by \cite[Proposition 7.4]{michor}. 
		Recall from \cite[Proposition 10.8]{michor} that the mapping $(\iota_V)_* \colon C^\infty_{\mathcal{D}} (U,V) \rightarrow C^\infty_{\mathcal{D}} (U,N)$ is an embedding of topological spaces.
			Further $(\iota_V)_* \circ R = \id_{C^\infty (U,N)}|_{\hat{C}_{\mathcal{D},\text{sub}}^\infty (U,V)}$. Since $\hat{C}_{\mathcal{D},\text{sub}}^\infty (U,V)$ is open by (1) and the identity is continuous, we deduce that $R$ is continuous.
		In conclusion $\res_{\mathcal{D}}$ is continuous.
		\item Clear from the definition as we can identify $O \subseteq J^r (U,V)$ via the identification $J^r(U,V) \subseteq J^r (M,N)$ with an open subset in $J^r (M,N)$.\qedhere
	\end{enumerate}
	 
\end{proof}

Our aim is now to prove the following folklore theorem. 
\begin{proposition}\label{prop: topologiescoincide}
 If $M$ and $N$ are finite-dimensional smooth manifolds, then the very strong and the $\mathcal{D}$-topology coincide. 
\end{proposition}

To prove the Proposition we need to prove the following special case first.

\begin{proposition}\label{prop: top:vscase}
 Let $U \subseteq \mathbb{R}^m$ and $V \subseteq \mathbb{R}^n$ be open subsets, then $\vsmooth (U,V) =  C^\infty_{\mathcal{D}} (U,V)$
\end{proposition}

\begin{proof}
        Comparing the bases of the $\mathcal{D}$-topology and vS-topology, it is clearly sufficient to show the following: For each elementary \( \mathcal{D} \)-neighborhood \( \mathcal{M} = \mathcal{M}^r (A,O) \) and \( f \in \mathcal{M} \), there exists an \( \epsilon > 0 \) such that \( \mathcal{N}^r (f;A,\epsilon) \subseteq \mathcal{M} \), and conversely that for every elementary vS-neighborhood \( \mathcal{N} = \mathcal{N}^r (f;A,\epsilon) \) and \( g \in \mathcal{N} \) there exists a finite cover of $A$ by compact sets $K_k$ and open subsets \( O_k \subseteq J^r (U,V) \) such that \( g \in \bigcap_{1\leq k \leq n}\mathcal{M}^r (K_k,O_k) \subseteq \mathcal{N} \).

        To this end, we fix \( r \in \mathbb{N} \), \( A \subseteq V \) a compact subset, and \( O \subseteq J^r (U,V) \) open. 
	Let us prove that \( \mathcal{M}^r (A,O) \) is open in the vS-topology. To this end consider \( f \in \mathcal{M}^r (A,O) \). 
	For all \( x \in A \) we have $j^rf(x) \in O$, whence there exist \( O_{x,U} \subseteq U \) open, an open ball \( B_{x,0} := B_{\epsilon_{x,0}'} (f(x)) \subseteq V \) and for every \( i \in \lbrace 1,\dots,r \rbrace \) open balls
	\begin{align*} 
	    B_{x,i} :=  \left\lbrace T \in L_{\text{sym}}^i (\RR^m , \RR^n) : \sup_{\Vert (y_1,\dots,y_i) \Vert \leq 1 } \left\Vert T(y_1,\dots,y_i) - \frac{\dd^{(i)} f(x;y_1,\dots,y_i)}{i!} \right\Vert < \epsilon_{x,i}' \right\rbrace 	 
	\end{align*}
       such that we have 
	\[
		j^r f(x) \in O_{x,U} \times \prod_{i=0}^r B_{x,i} \subseteq O \subseteq J^r(U,V).
	\]
	Recall that a map which is smooth in the sense that all partial derivatives of arbitrarily high order exist, is also smooth in the sense of \Frechet differentiability (see \cite[ 8.12.8]{Dieudonne}).
	Hence $j^rf \colon U \rightarrow J^r (U,V)$ is continuous and we can thus find for each $x$ a compact $x$-neighborhood $K_{x,U} \subseteq O_{x,U}$ such that 
	\begin{displaymath}\label{eq: inO}
		j^r f(y) \in O_{x,U} \times \prod_{i=0}^r B_{x,i} \subseteq O \subseteq J^r(U,V) \quad \text{for all } y \in K_{x,U} .
	\end{displaymath} 
	Since \( A \) is compact there exist \( x_1,\dots,x_n \in A \) with 
	\begin{equation}\label{eq: cov}
		j^r f(A) \subseteq \bigcup_{k=1}^n \left( K_{x_k,U} \times \prod_{i=0}^r B_{x_k,i} \right) \subseteq \bigcup_{k=1}^n \left( O_{x_k,U} \times \prod_{i=0}^r B_{x_k,i} \right) \subseteq O
	\end{equation}
	Set \( \epsilon' = \min \lbrace \epsilon_{x_k,i}' : 1 \leq k \leq n, 0 \leq i \leq r \rbrace \). 

        Exploiting the triangle inequality and multilinearity of the \Frechet derivative, one constructs constants \( \epsilon_{k,i} > 0 \) such that:	
        If \( g \colon U \to V \) is smooth and $\Vert \dd^{(i)} g (x ;\alpha_1,\dots,\alpha_i) - \dd^{(i)} f (x ;\alpha_1,\dots,\alpha_i) \Vert < \epsilon_{k,i}$ for all \( \alpha_1,\dots,\alpha_i \in \lbrace e_1,\dots,e_m \rbrace \) and $x \in K_{x_k,U}$, then 
	\begin{equation}
		\label{inDball}
		\left\Vert \frac{\dd^{(i)} g(x ;y_1,\dots,y_i)}{i!} - \frac{\dd^{(i)} f(x;y_1,\dots,y_i)}{i!} \right\Vert < \epsilon'  
	\end{equation}
	for $x \in A,$ \( y_1, \dots, y_i \in \RR^m \) with \( \Vert (y_1,\dots,y_i) \Vert \leq 1 \). 
	Set \( \epsilon = \min \lbrace \epsilon_{k,i} : 1 \leq k \leq n, 0 \leq i \leq r \rbrace \). If \( g \in \mathcal{N}^r (f;A,\epsilon) \), then for \( 1 \leq k \leq n \) and \( 0 \leq i \leq r \), we have for $x \in A$
	\[
		\Vert \dd^{(i)} g (x ;\alpha_1,\dots,\alpha_i) - \dd^{(i)} f (x ;\alpha_1,\dots,\alpha_i) \Vert \leq \Vert g - f \Vert (r,A) < \epsilon \leq \epsilon_{k,i},
	\]
	which implies \eqref{inDball}. By construction of the neighborhoods, we derive from \eqref{eq: cov} that \( g \in \mathcal{M}^r (A,O) \). Hence \( \mathcal{N}^r (f;A,\epsilon) \subseteq \mathcal{M}^r (A,O) \) and $ \mathcal{M}^r (A,O)$ is open in the vS-Topology.
	
	For the converse, fix \( g \in \mathcal{N}^r (f;A,\epsilon) \). 
	Set $\delta = \frac{\epsilon - \Vert g - f \Vert (r,A)}{2}$ and define for $1\leq i \leq r$ and $x \in A$ the open sets 
	\begin{align*}
	 B_{i,x} := \left\{ T \in L^i_{\text{sym}} (\RR^m,\RR^n) : \sup_{\lVert(y_1,\ldots,y_i)\rVert \leq 1} \left\Vert T(y_1,\ldots,y_i)-\frac{d^{(i)}g(x;(y_1,\ldots,y_i))}{i!}\right\Vert < \delta\right\} 
	\end{align*}
        Since $d^{(i)}g \colon U \rightarrow L^i_{\text{sym}} (\RR^m,\RR^n)$ is continuous (cf.\ \cite[8.12.8]{Dieudonne}), there are finitely many $x_k \in A, 1\leq k \leq n$ and compact $x_k$-neighborhoods $K_k$ such that $d^{(i)} g (K_k) \subseteq B_{i,x_k}$ for each $1\leq i \leq r$. 
        Dividing the compact sets $K_k$ into smaller parts, we may assume that there are open sets $B_{0,x_k}$ such that $g(K_k) \subseteq B_{0,x_k}$ and $z \in B_{0,x_k}$ implies $\Vert z-g(x)\Vert < \delta$ for all $x \in K_k$.
        Now define $O_k := U \times \prod_{i=0}^r B_{i,x_k}$ and observe that this set is open in $J^r(U,V)$. 
        Further we have $j^r g (K_k) \subseteq O_k$ for each $k$ and thus $g \in \bigcap_{1\leq k\leq n} \mathcal{M}^r(K_k,O_k)$.
        Exploiting the triangle inequality, we derive for $h \in  \bigcap_{1\leq k \leq n} \mathcal{M}^r(K_k,O_k)$ the estimate 
        \begin{align*}
         \Vert f-h\Vert(r,A) &\leq \Vert f-g\Vert(r,A) + \Vert g-h\Vert(r,A) \leq \Vert f-g\Vert(r,A) +  \sup_{1\leq k \leq n} \Vert g-h\Vert (r,K_k)\\
			    &= \Vert f-g\Vert(r,A) +  \sup_{1\leq k \leq n} \sup_{x\in K_k}  \sup_{0 \leq |\alpha| \leq r} \Vert d^{(i)} g(x;\alpha)-d^{(i)}h(x,\alpha)\Vert\\
			    &\leq \Vert f-g\Vert(r,A) +  \sup_{1\leq k \leq n} \sup_{x\in K_k} \sup_{0\leq i \leq r} \Vert d^{(i)} g(x;\cdot)-d^{(i)}h(x,\cdot)\Vert_{\text{op}}\\
			    &< \Vert f-g\Vert(r,A) + \sup_{1\leq k \leq n} \delta = \epsilon.
        \end{align*}
        The last inequality is derived from the definition of $O_k$ (or $B_{i,x_k}$, respectively).
        Summing up, $h \in \mathcal{N}^r (f;A,\epsilon)$ and we see that $\mathcal{N}^r (f;A,\epsilon)$ is open in the $\mathcal{D}$-topology.
\end{proof}

\begin{proof}[Proof of Proposition \ref{prop: topologiescoincide}]
	Observe that by definition of the very strong topology and the definition of the $\mathcal{D}$-topology it suffices to prove that every elementary neighborhood $\mathcal{N}^r (f; A,(U,\phi),(V,\psi),p,\epsilon )$ is the union of $\bigcap_{1\leq k \leq n}\mathcal{M}^r (K_k,O)$, where the $K_k$ are a finite family of compact sets which cover $A$ (cf.\ proof of Proposition \ref{prop: top:vscase}). Then $\mathcal{N}^r (f; A,(U,\phi),(V,\psi),p,\epsilon )$ is open and we can write each basic neighborhood in $\vsmooth (M,N)$ as a (possibly infinite) union of sets of the form $\Omega (\mathbf{L},\mathbf{U})$ (cf.\ Definition \ref{defn: Dtop}).
	
	Hence we fix $\mathcal{N}^r (f; A,(U,\phi),(V,\psi),p,\epsilon )$. 
	If necessary we shrink $U$ to achieve $f(U) \subseteq V$ (while still $A \subseteq U$).
	Thus we obtain a corresponding elementary neighborhood $$\mathcal{N}^r (f|_U^V; A,(U,\phi),(V,\psi),p,\epsilon ) \subseteq \vsmooth (U,V).$$
	Clearly $\res_{\text{vS}}^{-1} (\mathcal{N}^r (f|_U^V; A,(U,\phi),(V,\psi),p,\epsilon ))= \mathcal{N}^r (f; A,(U,\phi),(V,\psi),p,\epsilon ) $.
	Now consider the commutative diagram:
	\begin{equation}\label{diag: commute} \begin{gathered}
	\xymatrix{
		C_{\text{vS,sub}}^{\infty} (U,V) \ar[r]^{\res_{\text{vS}}} \ar@{=}[d] & \vsmooth (U,V) \ar[rr]_{\cong}^{(\phi^{-1})^* \circ \psi_*} \ar@{=}[d] & & \vsmooth (\phi(U),\psi(V)) \ar@{=}[d]_{\cong} \\
		C_{\mathcal{D},\text{sub}}^{\infty} (U,V) \ar[r]^{\res_{\mathcal{D}}} & C_{\mathcal{D}}^{\infty} (U,V) \ar[rr]_{\cong}^{(\phi^{-1})^* \circ \psi_*}     & & C_{\mathcal{D}}^{\infty} (\phi(U),\psi(V))		
	}\end{gathered}	\end{equation}
	Since both instances of \( (\phi^{-1})^* \circ \psi_* \) are homeomorphisms, and the rightmost identity map is a homeomorphism by Proposition \ref{prop: topologiescoincide}, the middle identity map is also a homeomorphism. 
        From the second part of the proof of Proposition \ref{prop: topologiescoincide} we derive that $\mathcal{N}^r (f|_U^V; A,(U,\phi),(V,\psi),p,\epsilon )$ indeed is a union of open sets of the form $\bigcap_{1 \leq k \leq n}\mathcal{M}^r (K_k,O_k)$ in $C^\infty_{\mathcal{D}} (U,V)$.
        Now Lemma \ref{biglemmaD} (3) implies together with the commutativity of \eqref{diag: commute} that also $\mathcal{N}^r (f; A,(U,\phi),(V,\psi),p,\epsilon )$ is a union of neighborhoods $\bigcap_{1\leq k \leq n}\mathcal{M}^r (K_k,O)$ in $C^\infty_{\mathcal{D}} (M,N)$.
        We conclude that the very strong topology is coarser than the $\mathcal{D}$-topology.
        
	For the converse observe again, that by the definitions of the topologies it suffices to prove that every elementary $\mathcal{D}$-neighborhood $\mathcal{M}^r (A,O)$ is the union of elementary neighborhoods $\mathcal{N}^r (f; A,(U,\phi),(V,\psi),p,\epsilon )$. 
	Here we have invoked Lemma \ref{lem: hacking} to see that it is indeed enough to consider the case $A \subseteq (U,\phi)$ and $O \subseteq J^r (U,V)$ for some charts $(U,\phi)$ and $(V,\psi)$.
	Hence, one can argue as in the first case, if one replaces Lemma \ref{biglemmaD} (3) with Lemma \ref{biglemma} (3).
	
	Summing up, the $\mathcal{D}$-topology is also coarser than the very strong topology, whence they coincide on $C^\infty (M,N)$.
\end{proof}

 Note that we have not used in any essential way in this Appendix that the target manifold $N$ is a finite-dimensional manifold. 
 Indeed, all arguments readily generalize without changes to the case where $N$ is a Banach manifold (and the source $M$ is finite-dimensional). 
 In \cite{michor} only finite-dimensional manifolds are considered as target.
 As the definition of jet bundles (and their topology) generalizes verbatim to the realm of Banach manifolds, one could also define a $\mathcal{D}$-topology for the space $C^\infty (M,N)$ where $N$ is a Banach manifold.
 We conclude that the statement of Proposition \ref{prop: topologiescoincide} remains true even if $N$ is a Banach manifold (and $M$ finite-dimensional).
 Hence we can copy the following results from \cite{michor} by appealing to Proposition \ref{prop: topologiescoincide}:
 
 \begin{corollary}
  Let $M$ be a finite-dimensional manifold and $N$ be a Banach manifold. 
  Then the topological space $\vsmooth (M,N)$ is a Baire space. 
 \end{corollary}
 
 \begin{proof}
  Combine \cite[4.7.6]{michor} with Proposition \ref{prop: topologiescoincide}. Note that the proof in \cite{michor} carries over verbatim to the case of $N$ being an infinite-dimensional Banach manifold.
 \end{proof}

\begin{remark}
 In \cite{michor} a refinement of the $\mathcal{D}$-topology, the so called $\mathcal{FD}$-topology is constructed. 
 The construction is similar to the construction of the fine very strong topology, whence Proposition \ref{prop: topologiescoincide} implies that the fine very strong topology and the $\mathcal{FD}$-topology on $C^\infty (M,N)$ coincide for finite-dimensional manifolds $M$ and $N$.
 
 In particular, \cite[Remark 4.11]{michor} thus implies that $\fsmooth (M,N)$ will in general not be a Baire space. 
 However by loc.cit.,  we then derive that for $M$ and $N$ finite-dimensional the topological space $\fsmooth (M,N)$ is paracompact and normal.
\end{remark}

\def\polhk#1{\setbox0=\hbox{#1}{\ooalign{\hidewidth
  \lower1.5ex\hbox{`}\hidewidth\crcr\unhbox0}}}

\end{document}